\documentclass[oneside]{amsart}

\usepackage{amsthm,amssymb}
\usepackage[abbrev]{amsrefs}
\usepackage{amsmath,amssymb}
\usepackage{amsthm}
\usepackage[abbrev]{amsrefs}
\usepackage{latexsym}

\usepackage{bm}
\usepackage{graphicx, xcolor}
\usepackage{hyperref}

\usepackage{amsmath, mathtools, amssymb}
\usepackage{amsthm}
\usepackage{mathrsfs}
\usepackage[mathscr]{eucal}
\usepackage{tikz}
\usetikzlibrary{decorations.pathmorphing}
\usetikzlibrary{matrix, arrows}
\usepackage[abbrev]{amsrefs}
\usepackage{latexsym}
\usepackage{stmaryrd}

\usepackage[top=25truemm,bottom=20truemm,left=20truemm,right=20truemm]{geometry}
\definecolor{eldritch}{RGB}{205,18,29}

\numberwithin{equation}{section}
\newtheorem{thm}{Theorem}[section]
\newtheorem{prop}[thm]{Proposition}
\newtheorem{lem}[thm]{Lemma}
\newtheorem{cor}[thm]{Corollary}

\theoremstyle{definition}
\newtheorem{dfn}[thm]{Definition}
\newtheorem{ex}[thm]{Example}
\newtheorem{rem}[thm]{Remark}
\newtheorem{obs}[thm]{Observation}

\newcommand{\defn}[1]{\it{#1}\rm{}}

\newcommand{\Def}[1]{\it{#1}\rm{}}

\newcommand{\Z}{\mathbb{Z}}

\newcommand{\R}{\mathbb{R}}
\newcommand{\Q}{\mathbb{Q}}
\newcommand{\ep}{\varepsilon}
\newcommand{\Hom}[1]{\mathrm{Hom}_{#1}}

\newcommand{\opposite}{^{\mathrm{op}}}
\newcommand{\id}[1]{\mathrm{id}_{#1}}
\newcommand{\proj}[1]{\mathrm{pr}_{#1}}

\newcommand{\Set}{\mathsf{Set}}

\newcommand{\sSet}{\mathsf{sSet}}

\newcommand{\D}{\bm{\Delta}}

\newcommand{\boundary}[1]{\partial\Delta[#1]}

\newcommand{\simplex}[2]{\Delta^{#1}[#2]}
\newcommand{\topsimplex}[1]{\Delta_{#1}}
\newcommand{\diffsimplex}[1]{\Delta^{#1}}
\newcommand{\diffhorn}[2]{\Lambda^{#2}_{#1}}

\newcommand{\K}{\bm{k}}

\newcommand{\dR}[3]{\mathcal{A}^{#2}({#1}{#3})}

\newcommand{\dq}{\vartheta}

\newcommand{\bs}[1]{\mathsf{b}_{#1}}
\newcommand{\fs}[1]{\mathsf{f}_{#1}}
\newcommand{\us}[1]{\mathsf{u}_{#1}}
\newcommand{\cact}[2]{^{#1}_{#2}}
\newcommand{\fint}[1]{{#1}_{\ast}}
\newcommand{\bint}[1]{\partial{#1}_{\ast}}

\begin{document}
\title{On Iterated Integral on Simplicial Sets}
\author{Ryohei Kageyama}
\address{Mathematical Institute, Tohoku University, Sendai 980-8578, Japan}
\email{ryohei.kageyama.s8@dc.tohoku.ac.jp}
%\date{\Large 20?? 年 1 月} %年・月を入れる. 
\date{\today}
\maketitle
\begin{abstract}
A simplicial analogy of Chen's iterated integral was introduced in \cite{kageyama2022higher}. However, its properties were hardly investigated in \cite{kageyama2022higher}. In particular, no mention is made of whether it coincides with Chen's iterated integral as a special case. In this paper, we answer this question.

This paper consists of two main parts. One of them is the research of elementary properties of simplicial iterated integral. In particular, we describe a relationship between iterated integral and homotopy pullback.

The other one is a comparison of Chen's iterated integral and simplicial iterated integral.

At the end of this paper, we observe Chen's theorem and Hain's theorem which connects rational homotopy groups and (co)homologies of a smooth manifold, and give an $\infty$-categorical viewpoint.
\end{abstract}
\tableofcontents %目次
%\setcounter{section}{-1}
%\section{On This Research, Brief Reviews and Notations}
\section{Introduction}
A simplicial analogy of Chen's iterated integral was introduced in \cite{kageyama2022higher} for the study of higher holonomies via simplicial viewpoints. It is known that classical holonomy representation can be constructed using Chen's iterated integral. It is also known that $2$-holonomy (that is a $2$-functor from path $2$-category to some strict $2$-category) is also can be constructed using Chen's iterated integral. On the other hand, it is known that all ``homotopical data'' of topological space is contained in the singular simplicial set. For example, fundamental groupoid coincides with the homotopy category of the singular simplicial set. Therefore we would like to develop simplicial viewpoints of higher holonomy. 
It is a motivation to construct an analogy of iterated integral. to simplicial sets. However, its properties were hardly investigated in \cite{kageyama2022higher}. In particular, no mention is made of whether it coincides with Chen's iterated integral as a special case. In this paper, we answer this question.

This paper consists of two main parts. One of them is the research of elementary properties of simplicial iterated integral. In particular, we describe a relationship between iterated integral and homotopy pullback. In the section \ref{review sii}, we review definitions of the simplicial analogy of fiberwise integration and the analogy of iterated integral. The review is also required in the later part. In the section \ref{property sii}, we show that the iterated integral is a morphism of commutative cochain algebras induced by the universality of homotopy pushout.

The other one is a comparison of Chen's iterated integral and simplicial iterated integral. In section \ref{review diff}, we review the loop space of smooth manifold as diffeological space. Since we use Kihara's diffeology to connect diffeological spaces and simplicial sets, we also review it. In \ref{review dR}, we describe some comparison of some types of de Rham algebras. We also describe de Rham theorem for diffeological spaces which is a result of Kuribayashi \cite{kuribayashi2020simplicial}. The comparison of Chen's iterated integral and simplicial iterated integral is described in \ref{comparison}.

At the end of this paper, we observe Chen's theorem and Hain's theorem which connects rational homotopy groups and (co)homologies of a smooth manifold. Combining the theorems in this paper, we give a claim that we believe is a generalization of Chen's theorem and Hain's theorem. The claim gives an $\infty$-categorical viewpoint.

\subsection*{Notation}
In this paper, we use the following notations:
\begin{itemize}
	\item For any (graded) module $V$ over some commutative ring $\K$, we denote its (graded) tensor algebra $\bigoplus_{r}V^{\otimes_{\K}r}=\K\oplus V\oplus(V\otimes_{\K}V)\oplus\cdots$ over $\K$ by $\mathsf{T}V_{\K}$.
	\item For each order preserving map $\alpha\colon[m]\to[n]$, also denote the induced map $\simplex{}{m}\to\simplex{}{n}$ by $\alpha$.
	\item For each pair of simplcial set $(X,Y)$, we denote the function complex as $Y^{X}$ or $\mathscr{F}\mathit{un}(X,Y)$.
	\item We call a simplicial map $f\colon X\to Y$ is a \defn{weak homotopy equivalence} if the induced map $f^{\ast}\colon\pi_{0}(K^{Y})\to\pi_{0}(K^{X})$ is bijective for any \bf{}Kan complex\rm{} $K$.
%	\item We call a simplicial map $f\colon X\to Y$ is a \defn{weak categorical equivalence} if the induced map $f^{\ast}\colon\tau_{0}(\mathscr{C}^{Y})\to\tau_{0}(\mathscr{C}^{X})$ is bijective for any \bf{}quasi-category\rm{} $\mathscr{C}$, where $\tau_{0}\mathscr{C}$ is a set of connected components of homotopy category $\tau_{1}\mathscr{C}$ of quasi-category $\mathscr{C}$.
	\item We often denote cofibration as $\hookrightarrow$.
	\item We often denote fibration as $\twoheadrightarrow$.
	\item We often denote a weak equivalence as $\xrightarrow{\sim}$. On the other hand, we often denote an isomorphism as $\xrightarrow{\simeq}$.
\end{itemize}
\subsection*{Acknowlegdments}
The author would like to thank Yuji Terashima, Shun Wakatsuki, Takumi Maegawa and Shun Oshima for useful communication.

\section{elementary properties of (iterated) integral on a simplicial set}\label{review sii}
\subsection{review of integration on a simplicial set}
\subsubsection{Divided Power de Rham Complexes}
An order-preserving map $\alpha\colon[m]\to[n]$ gives an affine map $\alpha_{\ast}\mathbb{A}_{\R}^{m+1}\to\mathbb{A}_{\R}^{n+1}$
\begin{align*}
	(x_{0},\dots,x_{m})\mapsto(\sum_{\alpha(j)=0}x_{j},\dots,\sum_{\alpha(j)=n}x_{j}),
\end{align*}
and an affine map $\bm{V}(\underset{0\leq i\leq m}{\sum}x_{i}-1)\to\bm{V}(\underset{0\leq i\leq n}{\sum}x_{i}-1)$ between hyperplanes. It induces a map between subspaces
\begin{center}
\begin{tikzpicture}[auto]
	\node (11) at (0, 1.2) {$\diffsimplex{m}$};
	\node (12) at (4, 1.2) {$\diffsimplex{n}$};
	\node (21) at (0, 0) {$\bm{V}(\underset{0\leq i\leq m}{\sum}x_{i}-1)$};
	\node (22) at (4, 0) {$\bm{V}(\underset{0\leq i\leq n}{\sum}x_{i}-1)$};
	\path[draw, right hook->] (11) -- (21);
	\path[draw, right hook->] (12) -- (22);
	\path[draw, ->, dashed] (11) --node {$\scriptstyle \alpha_{\ast}$} (12);
	\path[draw, ->] (21) --node {$\scriptstyle \alpha_{\ast}$} (22);
\end{tikzpicture}
\end{center}
which defined as $\diffsimplex{n}\coloneqq\{(x_{0},\dots,x_{n})\in\bm{V}(\sum_{i}x_{i}-1)|x_{i}\in[0,1]\}$ for each $n\geq0$. For each $n\geq0$, there is an isomorphim $\mathbb{A}_{\R}^{n}\cong\bm{V}(\sum_{i}x_{i}-1)$ defined as follows:
\begin{align*}
	\mathbb{A}_{\R}^{n}\to\bm{V}(\sum_{i}x^{i}-1),\ &\ (t^{1},\dots,t^{n})\mapsto(1-t^{1},t^{1}-t^{2},\dots,t^{n-1}-t^{n},t^{n}-0)\\
	\bm{V}(\sum_{i}x^{i}-1)\to\mathbb{A}_{\R}^{n},\ &\ (x^{0},\dots,x^{n})\mapsto(\sum_{i=1}^{n}x^{i},\dots,\sum_{i=n}^{n}x^{i}).
\end{align*}
The image of $\diffsimplex{n}$ under the isomorphism is given by
\begin{align*}
	\topsimplex{n}\coloneqq\{(t_{1},\dots,t_{n})|1\geq t_{1}\geq\dots\geq t_{n}\geq0\}.
\end{align*}
For each order-preserving map $\alpha\colon[m]\to[n]$, we obtain a commutative diagram
\begin{center}
\begin{tikzpicture}[auto]
	\node (11) at (0, 1.2) {$\mathbb{A}_{\R}^{m}$};
	\node (21) at (0, 0) {$\mathbb{A}_{\R}^{n}$};
	\node (12) at (2.5, 1.2) {$\bm{V}(\sum_{j}X^{j}-1)$};
	\node (22) at (2.5, 0) {$\bm{V}(\sum_{j}X^{j}-1)$};
	\node (13) at (5, 1.2) {$\mathbb{A}_{\R}^{m+1}$};
	\node (23) at (5, 0) {$\mathbb{A}_{\R}^{n+1}$};
	\path[draw, ->] (11) --node {$\scriptstyle \simeq$} (12);
	\path[draw, right hook->] (12) -- (13);
	\path[draw, ->] (21) --node {$\scriptstyle \simeq$} (22);
	\path[draw, right hook->] (22) -- (23);
	\path[draw, ->] (11) --node[swap] {$\scriptstyle \alpha_{\ast}$} (21);
	\path[draw, ->] (12) --node {$\scriptstyle \alpha_{\ast}$} (22);
	\path[draw, ->] (13) --node {$\scriptstyle \alpha_{\ast}$} (23);
\end{tikzpicture}.
\end{center}
Where $\alpha_{\ast}\colon\mathbb{A}_{\R}^{m}\to\mathbb{A}_{\R}^{n}$ is defined as follows:
\begin{align*}
	\proj{i}\alpha_{\ast}(t_{1},\dots,t_{n})\coloneqq
		\begin{cases}
			t_{\min\{j\in[m]|\alpha(j)\geqslant i\}}&(\alpha(m)\geqslant i)\\
			0&(\alpha(m)<i)
		\end{cases}.
\end{align*}
An affine space $\mathbb{A}_{\Q}^{n}$ corresponds to a polynomial ring $\Q[t_{1},\dots,t_{n}]$ and a hyperplane $\bm{V}(\sum_{i}x_{i}-1)\subset\mathbb{A}_{\Q}^{n+1}$ corresponds to a quotient ring $\Q[x_{0},\dots,x_{n}]/(\sum_{i}x_{i}-1)$.  In addition, the isomorphism $\mathbb{A}_{\Q}^{n}\colon\bm{V}(\sum_{i}x_{i}-1)$ corresponds to a ring isomorphism
\begin{align*}
	\mathbb{Q}[t_{1},\dots,t_{n}]\cong\mathbb{Q}[x_{0},\dots,x_{n}]/(\sum_{i}x_{i}-1).
\end{align*}
The quotient ring $\mathbb{Q}[x_{0},\dots,x_{n}]/(\sum_{i}x_{i}-1)$ just coincide with a ring whose elements are (Sullivan's) differential $0$-form on an  $n$-dimensional standard simplex $\simplex{}{n}$. Therefore it is not unnatural to regard the polynomial ping $\mathbb{Q}[t_{1},\dots,t_{n}]$ as a ring of functions on an $n$-dimensional standard simplex $\simplex{}{n}$.

However, the de Rham complex (which corresponds to this ring) has trivial torsion (as Abelian group). In addition, we must assume the character of the ring we are considering is $0$. Therefore we use a ring that does not contain $\Q$. The most extreme candidate is $\Z$, in which case ``integration'' cannot be defined. So we consider a divided power polynomial algebra over $\Z$, that is a free Abelian group
\begin{align*}
	\Z\langle x_{0},\dots,x_{n}\rangle\coloneqq\bigoplus_{N_{0},\dots,N_{n}\geq0}\Z x_{0}^{[N_{0}]}\dots x_{n}^{[N_{n}]}
\end{align*}
with product defined as
\begin{align*}
	(x_{0}^{[N_{10}]}\dots x_{n}^{[N_{1n}]})(x_{0}^{[N_{20}]}\dots x_{n}^{[N_{2n}]})=\frac{(N_{10}+N_{20})}{N_{10}!N_{20}!}\dots\frac{(N_{1n}+N_{2n})}{N_{1n}!N_{2n}!}x_{0}^{[N_{10}+N_{20}]}\dots x_{n}^{[N_{1n}+N_{2n}]}
\end{align*}
to be the ring of ``functions on an $n$-dimensional standard simplex $\simplex{}{n}$''. We denote $x_{i}^{[1]}$ as $x_{i}$. This ring can be embedded in the polynomial ring $\Q[x_{0},\dots,x_{n}]$ by the canonical way which is given by following morphism :
\begin{align*}
	x_{0}^{[N_{0}]}x_{1}^{[N_{1}]}\dots x_{n}^{[N_{n}]}\mapsto\frac{1}{N_{0}!N_{1}!\dots N_{n}!}x_{0}^{N_{0}}x_{1}^{N_{1}}\dots x_{n}^{N_{n}}.
\end{align*}
Similarly, three kinds of (canonical) morphisms
	\begin{align*}
		\Z\langle x_{0},\dots,x_{n}\rangle&\to\Q\langle x_{1},\dots,x_{n}\rangle\\
		\Z\langle x_{0},\dots,x_{n}\rangle&\to\Z_{(p)}\langle x_{1},\dots,x_{n}\rangle\\
		\Z\langle x_{0},\dots,x_{n}\rangle&\to\Z\langle x_{1},\dots,x_{n}\rangle
	\end{align*}
	are given as follows where $p$ is a prime number:
	\begin{align*}
		x_{0}^{[N_{0}]}x_{1}^{[N_{1}]}\dots x_{n}^{[N_{n}]}
			&\mapsto
		\frac{1}{N_{0}!}x_{1}^{[N_{1}]}\dots x_{n}^{[N_{n}]},\\
		x_{0}^{[N_{0}]}x_{1}^{[N_{1}]}\dots x_{n}^{[N_{n}]}
			&\mapsto
		\frac{1}{N_{0}!}p^{N_{0}}x_{1}^{[N_{1}]}\dots x_{n}^{[N_{n}]},\\
		x_{0}^{[N_{0}]}x_{1}^{[N_{1}]}\dots x_{n}^{[N_{n}]}
			&\mapsto
		\begin{cases}
			x_{1}^{[N_{1}]}\dots x_{n}^{[N_{n}]}&(N_{0}=0)\\
			0&(N_{0}\neq0)
		\end{cases}.
	\end{align*}
	More generally, a divided power polynomial algebra has a universal property like polynomial rings. Therefore, for each map $\ep\colon\{x_{0},x_{1},\dots,x_{n}\}\to\{x_{0},x_{1},\dots,x_{n}\}$, there exists a unique morphism $\overline{\ep}\colon\Z\langle x_{0},\dots,x_{n}\rangle\to\Z\langle x_{0},\dots,x_{n}\rangle$ satisfies $\overline{\ep}(x_{i}^{[N_{i}]})=\ep(x_{i})^{[N_{i}]}$ for each $i=0,\dots,n$.

We define a morphism $\alpha^{\ast}\colon\Z\langle x_{0},\dots,x_{n}\rangle\to\Z\langle x_{0},\dots,x_{m}\rangle$ as
\begin{align*}
	\alpha^{\ast}(x_{i}^{[N]})\coloneqq
		\begin{cases}
			x_{\min\{j|\alpha(j)\geq i\}}^{[N]}&(\alpha(m)\geq i)\\
			0&(\alpha(m)<i)
		\end{cases}
\end{align*}
for each order-preserving maps $\alpha\colon[m]\to[n]$. We obtain a simplicial $\Z\langle x_{0}\rangle$-algebra $\dR{\simplex{}{-}}{0}{,\Z\langle x_{0}\rangle}$ by above.
Hereafter we denote $x_{0}$ of these rings as $\dq$, and consider $\dq$ to be an element like the unit of the ring.

For each non-negative integer $n\geq0$ and arbitrary $\dR{\simplex{}{n}}{0}{,\Z\langle \dq\rangle}$-modules $M$, an (Abelian) group morphism of $\theta\colon\dR{\simplex{}{n}}{0}{,\Z\langle \dq\rangle}\to M$ which satisfies the following is called a divided power $\Z\langle\dq\rangle$-derivation:
\begin{align*}
	\theta(a)
		&=
	0&
		&\text{for all }
	a\in\Z\langle\dq\rangle,\\
	\theta(fg)
		&=
	g\theta(f)+f\theta(g)&
		&\text{for all }
	f,g\in\dR{\simplex{}{n}}{0}{,\Z\langle \dq\rangle},\\
	\theta(x_{i}^{[N]})
		&=
	x_{i}^{[N-1]}\theta(x_{i})&
		&\text{for all }
	i=1,\dots,n\text{ and }N\geq1.
\end{align*}
The $\dR{\simplex{}{n}}{0}{,\Z\langle \dq\rangle}$-module of divided power $\Z\langle\dq\rangle$-derivations of $\dR{\simplex{}{n}}{0}{,\Z\langle \dq\rangle}$ into $M$ gives a representable functor $\mathrm{Der}_{\Z\langle\dq\rangle}(\dR{\simplex{}{n}}{0}{,\Z\langle \dq\rangle},-)\colon\mathrm{Mod}_{\dR{\simplex{}{n}}{0}{,\Z\langle \dq\rangle}}\to\mathrm{Mod}_{\dR{\simplex{}{n}}{0}{,\Z\langle \dq\rangle}}$.
It is represented by a free $\Z\langle\dq,x_{1},\dots,x_{n}\rangle$-module $\dR{\simplex{}{n}}{1}{,\Z\langle \dq\rangle}$ generated by formal elements $dx_{1},\dots,dx_{n}$. In addition the derivation $d^{0}\colon\dR{\simplex{}{n}}{0}{,\Z\langle \dq\rangle}\to\dR{\simplex{}{n}}{1}{,\Z\langle \dq\rangle}$ corresponding to the identity $\id{}\colon\dR{\simplex{}{n}}{1}{,\Z\langle \dq\rangle}\to\dR{\simplex{}{n}}{1}{,\Z\langle \dq\rangle}$ is given as follows:
\begin{equation*}
	d^{0}(\sum_{N_{1},\dots,N_{n}}f_{N_{1},\dots,N_{n}}x_{1}^{[N_{1}]}\dots x_{n}^{[N_{n}]})
		\coloneqq
	\sum_{i=1}^{n}(\sum_{N_{1},\dots,N_{n}}f_{N_{1},\dots,N_{n}}x_{1}^{[N_{1}]}\dots x_{i}^{[N_{i}-1]}\dots x_{n}^{[N_{n}]})dx_{i}.
\end{equation*}
We denote the derivation $\dR{\simplex{}{n}}{0}{,\Z\langle \dq\rangle}\to\dR{\simplex{}{n}}{0}{,\Z\langle \dq\rangle}$ corresponding to the ``standard dual base'' 
\begin{align*}
	\chi_{dx_{i}}
		\colon
	\dR{\simplex{}{n}}{1}{,\Z\langle \dq\rangle}
		&\to
	\dR{\simplex{}{n}}{0}{,\Z\langle \dq\rangle}\\
	\sum_{j}f_{j}dx_{j}
		&\mapsto
	f_{i}
\end{align*}
by $\frac{\partial}{\partial x_{i}}$.

They give a graded (commutative) $\dR{\simplex{}{n}}{0}{,\Z\langle \dq\rangle}$-algebra
\begin{equation*}
	\dR{\simplex{}{n}}{\bullet}{,\Z\langle \dq\rangle}
		\coloneqq
	\mathsf{Sym}\dR{\simplex{}{n}}{1}{,\Z\langle \dq\rangle}[1]
		=
	\dR{\simplex{}{n}}{0}{,\Z\langle \dq\rangle}
		\oplus
	\dR{\simplex{}{n}}{1}{,\Z\langle \dq\rangle}
		\oplus
	\dots
		\oplus
	\dR{\simplex{}{n}}{n}{,\Z\langle \dq\rangle}
\end{equation*}
and a degree $-1$ derivation $d\colon\dR{\simplex{}{n}}{\bullet}{,\Z\langle \dq\rangle}\to\dR{\simplex{}{n}}{\bullet}{,\Z\langle \dq\rangle}$. In other words, we obtain a commutative cochain algebra $\dR{\simplex{}{n}}{\bullet}{,\Z\langle \dq\rangle}$. We call the cochain algebra the \Def{divided power de Rham algebra on standard simplex $\simplex{}{n}$}. For each order-preserving map $\alpha\colon[m]\to[n]$, the $\Z\langle\dq\rangle$-algebra morphism $\alpha^{\ast}$ gives a cochain algebra morphism $\alpha^{\ast}\colon\dR{\simplex{}{n}}{\bullet}{,\Z\langle \dq\rangle}\to\dR{\simplex{}{m}}{\bullet}{,\Z\langle \dq\rangle}$. Therefore we obtain a simplicial commutative cochain algebra $\dR{\simplex{}{-}}{\bullet}{,\Z\langle \dq\rangle}\colon\D\opposite\to\mathsf{dgA}_{\Z\langle\dq\rangle}$.

\subsubsection{Integration on Standard Simplices}
To define the integration of formal differential forms, we observe the classical case, in other words, the integration of a polynomial function of real coefficients. For any integer $a\in\R$ and non-negative ingeter $N$, the following (redundant) equation holds:
\begin{align*}
	\int_{\alpha}^{\beta}a\frac{x^{N}}{N!}dx=a\frac{\beta^{N+1}}{(N+1)!}-a\frac{\alpha^{N+1}}{(N+1)!}
\end{align*}
\begin{dfn}(iterated integral of divided power polynomial functions)
	Let $f=\sum_{N_{1},\dots,N_{r}}m_{N_{1},\dots,N_{r}}x_{1}^{[N_{1}]}\cdots x_{r}^{[N_{r}]}$ be an $r$-variable divided power polynomial of integer coefficients, that is an element of $\Z\langle\dq,x_{1},\dots,x_{n}\rangle$. Then we define the \Def{iterated integration} of $f$
	\begin{align*}
		&\int_{\alpha_{p}}^{\beta_{p}}\dots\int_{\alpha_{1}}^{\beta_{1}}fdx_{i_{1}}\cdots dx_{i_{p}}
		&(\alpha_{1},\dots,\alpha_{p},\beta_{1},\dots,\beta_{p}\in\{\dq,x_{1},\dots,x_{r},0\})
	\end{align*}
	inductively as follows:
	\begin{align*}
		\int_{\alpha_{1}}^{\beta_{1}}fdx_{i_{1}}
			&\coloneqq
		\sum_{N_{1},\dots,N_{r}}m_{N_{1},\dots,N_{r}}x_{1}^{[N_{1}]}\cdots(\beta_{1}^{[N_{i_{1}}+1]}-\alpha_{1}^{[N_{i_{1}}+1]})\cdots x_{r}^{[N_{r}]},\\
		\int_{\alpha_{p}}^{\beta_{p}}\dots\int_{\alpha_{1}}^{\beta_{1}}fdx_{i_{1}}\cdots dx_{i_{p}}
			&\coloneqq
		\int_{\alpha_{p}}^{\beta_{p}}(\int_{\alpha_{p-1}}^{\beta_{p-1}}\dots\int_{\alpha_{1}}^{\beta_{1}}fdx_{i_{1}}\cdots dx_{i_{p-1}})dx_{i_{p}}.
	\end{align*}
\end{dfn}

For each pair $(n,r)$ of integers, there is a canonical partition
\begin{equation}\label{chain partition}
	[n]\times[r]\cong\bigcup_{\Gamma\colon[n+r]\hookrightarrow[n]\times[r]}[n+r].
\end{equation}
using maximal chains.
For any maximal chains $\Gamma\colon[n+r]\hookrightarrow[n]\times[r]$, we denote the map $\proj{1}\Gamma\colon[n+r]\to[n]$ (resp. $\proj{2}\Gamma\colon[n+r]\to[r]$) as $\Gamma_{\bs{}}$ (resp. $\Gamma_{\fs{}}$).  In addition, a maximal chain $\Gamma\colon[n+r]\hookrightarrow[n]\times[r]$ define two order-preserving maps $\bs{\Gamma}\colon[n]\to[n+r],\fs{\Gamma}\colon[r]\to[n+r]$ as follows:
\begin{align*}
	\bs{\Gamma}(i)
		&\coloneqq
	\min\{j\in[p]|\Gamma_{\bs{}}(j)=i\},\\
	\fs{\Gamma}(i)
		&\coloneqq
	\min\{j\in[p]|\Gamma_{\fs{}}(j)\geq i\},\\
	\us{\Gamma}(i)
		&\coloneqq
	\fs{\Gamma}(\min\{j\in\{1,\dots,p\}|\fs{\Gamma}(j)-j=\fs{\Gamma}(i)-i\})-1.
\end{align*}
\begin{rem}(geometrical meanings)
	The geometric realization of the nerve of a poset $[n]\times[r]$ is just the product of topological standard simplices $\topsimplex{n}\times\topsimplex{r}$. The above partition \ref{chain partition} is a canonical wap to partition of the space into topological standard simplices. The intersection of a fiber $\proj{\topsimplex{n}}^{-1}(\bm{x})\subset\topsimplex{n}\times\topsimplex{r}$ of projection $\proj{\topsimplex{n}}\colon\topsimplex{n}\times\topsimplex{r}\to\topsimplex{n}$ and the image of each embedding $\Gamma_{\ast}\colon\topsimplex{n+r}\to\topsimplex{n}\times\topsimplex{r}$ is given as follows:
	\begin{align*}
		\proj{\topsimplex{n}}^{-1}(\bm{x})\cap\mathrm{Im}\Gamma_{\ast}
			&\cong
		(\Gamma_{\bs{}})_{\ast}^{-1}(\bm{x})\\
			&=
		\{(t_{1},\dots,t_{n+r})\in\topsimplex{n+r}|t_{\min\{j|\Gamma_{\bs{}}(j)\geq i\}}=x_{i}\}\\
			&=
		\{(t_{1},\dots,t_{n+r})\in\topsimplex{n+r}|t_{\bs{\Gamma}(i)}=x_{i}\}.
	\end{align*}
	For each maximal chain $\Gamma\colon[n+r]\hookrightarrow[n]\times[r]$,
	$\mathrm{Im}\bs{\Gamma}\cap\mathrm{Im}\fs{\Gamma}=\{0\}$, $\mathrm{Im}\bs{\Gamma}\cup\mathrm{Im}\fs{\Gamma}=[n+r]$ hold. Hence we can regard
	\begin{itemize}
		\item $\bs{\Gamma}$ represents the ``base direction''.
		\item $\fs{\Gamma}$ represents the ``fiber direction''.
	\end{itemize}
	(The standard coordinate of $\topsimplex{n+r}$ can be split into two kinds of ``direction'', ``base direction'' and ``fiber direction''.) In addition the following holds:
	$$\proj{\topsimplex{n}}^{-1}(\bm{x})\cap\mathrm{Im}\Gamma_{\ast}\cong(\Gamma_{\bs{}})^{-1}(\bm{x})\cong\topsimplex{r_{1}}\times\dots\times\topsimplex{r_{\mathsf{n}_{\Gamma}}}.
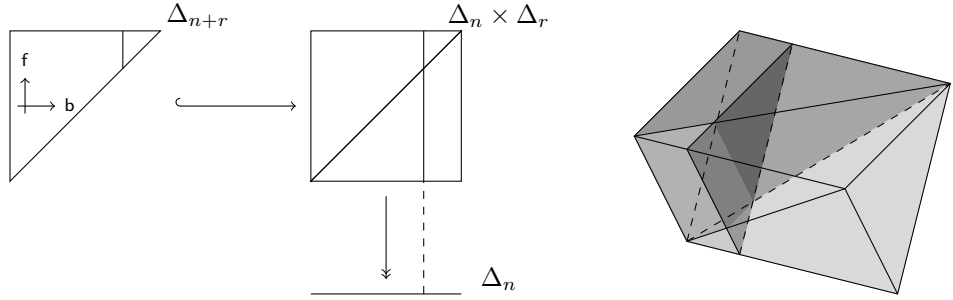
$$
	\begin{center}
	\begin{tikzpicture}[auto]
		\node (1) at (2.5, 3.7) {$\topsimplex{n+r}$};
		\node (2) at (6.5, 3.7) {$\topsimplex{n}\times\topsimplex{r}$};
		\node (3) at (6.5, 0.2) {$\topsimplex{n}$};
		\path[draw] (0, 1.5) -- (0, 3.5) -- (2, 3.5) -- cycle;
		\path[draw] (4, 1.5) -- (4, 3.5) -- (6, 3.5) -- cycle;
		\path[draw] (4, 1.5) -- (6, 1.5) -- (6, 3.5) -- cycle;
		\path[draw] (4, 0) -- (6, 0);
		\path[draw] (1.5, 3) -- (1.5, 3.5);
		\path[draw] (5.5, 1.5) -- (5.5, 3.5);
		\path[draw, dashed] (5.5, 0) -- (5.5, 1.5);
		\path[draw, right hook->] (2+0.2, 2.5) -- (4-0.2, 2.5);
		\path[draw, ->>] (5, 1.5-0.2) -- (5, 0+0.2);
		\node (b) at (0.8, 2.5) {$\scriptstyle \bs{}$};
		\node (f) at (0.2, 3.1) {$\scriptstyle \fs{}$};
		\path[draw, ->] (0.1, 2.5) -- (b);
		\path[draw, ->] (0.2, 2.4) -- (f);
		\draw[black!35, fill=black!35] (9, 0.7) -- (8.3, 2.1) -- (9.7, 3.5) -- (12.5, 2.8) --cycle;
		\draw[black!15, fill=black!15] (9, 0.7) -- (12.5, 2.8) -- (11.8, 0) --cycle;
		\draw[black!40, fill=black!40] (8.3, 2.1) -- (9.35, 2.275) -- (10.4, 3.325) -- (9.7, 3.5) --cycle;
		\draw[black!30, fill=black!30] (9, 0.7) -- (8.3, 2.1) -- (0.5*0.7+9, 3.25*0.7) -- (1.25*0.7+9, 1.75*0.7) -- (0.75*0.7+9, 1.25*0.7) -- cycle;
		\draw[black!20, fill=black!20] (9, 0.7) -- (0.75*0.7+9, 1.25*0.7) -- (1.25*0.7+9, 1.75*0.7) -- (9.7, 0.75*0.7) -- cycle;
		%%%
		\draw[black!60, fill=black!60] (10.25*0.7-9*0.7+9, 1.75*0.7) -- (9.5*0.7-9*0.7+9, 3.25*0.7) -- (11*0.7-9*0.7+9, 4.75*0.7) --cycle;
		\draw[black!40, fill=black!40] (9.7, 0.75*0.7) -- (0.75*0.7+9, 1.25*0.7) -- (1.25*0.7+9, 1.75*0.7) --cycle;
		\draw[black!50, fill=black!50] (0.75*0.7+9, 1.25*0.7) -- (9, 2.75*0.7) -- (0.5*0.7+9, 3.25*0.7) -- (1.25*0.7+9, 1.75*0.7) --cycle;
		\path[draw, dashed] (9, 1*0.7) -- (9.7, 5*0.7);
		\path[draw] (11.8, 0*0.7) -- (11.1, 2*0.7);
		\path[draw] (8.3, 3*0.7) -- (11.1, 2*0.7);
		\path[draw, dashed] (9, 1*0.7) -- (12.5, 4*0.7);
		\path[draw] (9.7, 0.75*0.7) -- (9, 2.75*0.7) -- (10.4, 4.75*0.7);
		\path[draw, dashed] (9.7, 0.75*0.7) -- (10.4, 4.75*0.7);
		\path[draw] (9, 1*0.7) -- (8.3, 3*0.7) -- (9.7, 5*0.7) -- (12.5, 4*0.7) -- (11.8, 0*0.7) --cycle;
		\path[draw] (8.3, 3*0.7) -- (12.5, 4*0.7);
		\path[draw] (9, 1*0.7) -- (11.1, 2*0.7) -- (12.5, 4*0.7);
	\end{tikzpicture}
	\end{center}
\end{rem}

\subsubsection{Fiberwise Integration}
Let $\Gamma\colon[n+r]\hookrightarrow[n]\times[r]$ be a maximal chain and $\omega\in\dR{\simplex{}{n+r}}{\bullet}{,\Z\langle\dq\rangle}$ be a differential form. Then there exists an essentially unique decomposition
\begin{align*}
	\omega
		=
	\sum_{i}\omega_{\Gamma,i,\fs{}}\wedge\Gamma_{\bs{}}^{\ast}\omega_{\Gamma,i,\bs{}}
\end{align*}
where $\omega_{\Gamma,i,\bs{}}$ is an element of $\dR{\simplex{}{n}}{\bullet}{,\Z\langle\dq\rangle}$ and $\omega_{\Gamma,i,\fs{}}$ is an element of $\dR{\simplex{}{n+r}}{\bullet}{,\Z\langle\dq\rangle}$ which  does not contain
$$x_{\bs{\Gamma}(1)}^{[N_{1}]},\dots,x_{\bs{\Gamma}(n)}^{[N_{n}]},dx_{\bs{\Gamma}(1)},\dots,dx_{\bs{\Gamma}(n)}.$$
In addition, there is a unique decomposition
\begin{equation*}
	\omega_{\Gamma,i,\fs{}}
		=
	\omega_{\Gamma,i,\fs{}}^{(r)}+\dots+\omega_{\Gamma,i,\fs{}}^{(0)}
\end{equation*}
where $\omega_{\Gamma,i,\fs{}}^{(j)}$ is an element of $\dR{\simplex{}{n}}{j}{,\Z\langle\dq\rangle}$. Especially there is a decomposition
\begin{equation*}
	\omega_{\Gamma,i,\fs{}}^{(p-(n+1))}
		=
	\sum_{j}f_{\Gamma,i,j}dx_{\fs{\Gamma}(1)}\wedge\dots\wedge dx_{\fs{\Gamma}(r)}
\end{equation*}
where $f_{\Gamma,i,,j}\in\dR{\simplex{}{n+r}}{0}{,\Z\langle\dq\rangle}$. Using this essentially unique representation, we obtain the following (where $x_{0}\coloneqq\dq$ and $x_{n+r+1}\coloneqq0$):
\begin{align*}
	\int_{\simplex{}{r}\cact{\Gamma}{}}\omega
		\coloneqq
	\sum_{i,j}\bs{\Gamma}^{\ast}(\int_{x_{\fs{\Gamma}(r)+1}}^{x_{\us{\Gamma}(r)}}\dots\int_{x_{\fs{\Gamma}(1)+1}}^{x_{\us{\Gamma}(1)}}f_{\Gamma,i,\lambda,j}dx_{\fs{\Gamma}(1)}\cdots dx_{\fs{\Gamma}(r)})\omega_{\Gamma,i,\bs{}}.
\end{align*}

Let $\omega\colon\simplex{}{n}\times\simplex{}{r}\to\dR{\simplex{}{-}}{\bullet}{,\Z\langle\dq\rangle}$ be a differential form on $\simplex{}{n}\times\simplex{}{r}$. It gives an $n$-simplex $\omega^{\wedge}$ of $\dR{\simplex{}{-}}{\bullet}{,\Z\langle\dq\rangle}^{\simplex{}{r}}$. Hence, by the Eilenberg-Zilber lemma, we obtain a unique decomposition $\omega=(\sigma\times\id{})^{\ast}\tilde{\omega}$ where $\sigma\colon[n]\to[m]$ is a surjection and $\tilde{\omega}^{\wedge}$ is a non-degenerate $m$-simplex of $\dR{\simplex{}{-}}{\bullet}{,\Z\langle\dq\rangle}^{\simplex{}{r}}$. Using this unique decomposition, we define as
\begin{align*}
	\fint{\proj{\simplex{}{n}}}\omega\coloneqq\sum_{\Gamma\colon[n+r]\hookrightarrow[n]\times[r]}\sigma^{\ast}(\int_{\simplex{}{r}\cact{\Gamma}{}}\Gamma^{\ast}\tilde{\omega}).
\end{align*}
\begin{lem}\label{Int Lem21}
	For each differential form $\omega\colon\simplex{}{n}\times\simplex{}{r}\to\dR{\simplex{}{-}}{\bullet}{,\Z\langle\dq\rangle}$ and order-preserving map $\alpha\colon[m]\to[n]$, the following holds:
	\begin{align*}
		\alpha^{\ast}\fint{\proj{\simplex{}{n}}}\omega=\fint{\proj{\simplex{}{m}}}((\alpha\times\id{})^{\ast}\omega).
	\end{align*}
\end{lem}
\begin{proof}
	See \cite{kageyama2022higher}.
\end{proof}

Let $\omega\colon X\times\simplex{}{r}\to\dR{\simplex{}{-}}{\bullet}{,\Z\langle\dq\rangle}$ be a differential form on $X\times\simplex{}{r}$. From Lemma \ref{Int Lem21}, we obtain a cocone
\begin{equation*}
	\fint{\proj{\simplex{}{\square}}}((-\times\id{})^{\ast}\omega)
		\colon
	\simplex{}{\square}
		\to
	\dR{\simplex{}{-}}{\bullet}{,\Z\langle\dq\rangle}
\end{equation*}
and obtain a differential form $\fint{\proj{X}}\omega\colon X\to\dR{\simplex{}{-}}{\bullet}{,\Z\langle\dq\rangle}$.
\begin{dfn}(simplicial integration)
	Let $\omega\colon X\times\simplex{}{r}\to\dR{\simplex{}{-}}{\bullet}{,\Z\langle\dq\rangle}$ be a differential form on $X\times\simplex{}{r}$. Then, the differential form $\fint{\proj{X}}\omega$ on $X$ is called \Def{the fiberwise integration of $\omega$ along the projection $\proj{X}\colon X\times\simplex{}{r}\to X$}.
\end{dfn}
\subsubsection{Stokes's theorem}
One of the important theorems for integrals on smooth manifolds is Stokes's theorem. This is a theorem that connects the integration of closed form with the integration on the boundary, and it follows that the integration gives a chain map from the de Rham complex to the singular cochain complex. We would like to consider this analogy for fiberwise integration on simplicial sets, but roughly speaking, the following obstacles exist:
\begin{itemize}
	\item The boundary of simplicial set $U$ is unknown in general.
	\item For example, the boundary of standard $2$-simplex $\simplex{}{2}$ is already known as $\boundary{2}$, but the integration $\fint{\proj{X}}(\omega|_{X\times\boundary{2}})$ does not coincide with what we seek.
\end{itemize}
The second problem is considered to be caused by the fact that, unlike the case of smooth manifolds, orientation is not taken into account. In light of simplicial homology, it is presumed that it is suitable to consider the linear combination $\sum_{i=0}^{n}(-1)^{i}\Delta\{0,\dots,\check{i},\dots,n\}$ as ``the boundary of standard $n$-simplex with orientation taken into account''. Since fiberwise integration on a simplicial set is the sum of integration on each simplex, we can consider the following ``integration''.
\begin{dfn}
	Let $\omega\colon X\times\simplex{}{r}\to\dR{\simplex{}{-}}{\bullet}{,\Z\langle\dq\rangle}$ be a formal differential form on $X\times\simplex{}{r}$. The differential form on $X$
	\begin{align*}
		\bint{\proj{X}}\omega
			\coloneqq
		\sum_{i=0}^{r}(-1)^{i}\fint{\proj{X}}((\id{}\times\delta_{i})^{\ast}\omega)
	\end{align*}
	is called \Def{the boundary fiberwise integration of $\omega$ along the projection $\proj{X}\colon X\times\simplex{}{r}\to X$}.
\end{dfn}
\begin{thm}\label{Stokes}
	Let $\omega\colon X\times\simplex{}{r}\to\dR{\simplex{}{-}}{\bullet}{,\Z\langle\dq\rangle}$ be a differential form on $X\times\simplex{}{r}$. Then, the following holds:
	\begin{equation*}
		\fint{\proj{X}}d\omega-\bint{\proj{X}}\omega
			=
		(-1)^{r}d\fint{\proj{X}}\omega
	\end{equation*}
\end{thm}
\begin{proof}
	See \cite{kageyama2022higher}.
\end{proof}

\subsection{Iterated Integral}
In this subsection, we define an analogy of Chen's iterated integral for simplicial sets and describe its elementary properties.
Let $X$ be a simplicial set and $\omega_{1},\dots,\omega_{r}\colon X\to\dR{\simplex{}{-}}{\bullet}{,\Z\langle\dq\rangle}$ be differential forms on $X$. Then we obtain a differential form on $X^{r}\coloneqq\underbrace{X\times\dots\times X}_{r}$ as $\proj{1}^{\ast}\omega_{1}\wedge\dots\wedge\proj{r}^{\ast}\omega_{r}$. It gives a differential form on $(X^{\simplex{}{1}})^{r}\times\simplex{}{1}^{r}$ by using a counit $\mathrm{ev}\colon\simplex{}{1}\times(-)^{\simplex{}{1}}\Rightarrow\id{}$ of the adjoint pair $\simplex{}{1}\times-\dashv(-)^{\simplex{}{1}}$. In addition, by using a simplicial map $\iota_{r}\colon\simplex{}{r}\to\simplex{}{1}^{r}$ obtained from an order-preserving map $[r]\to[1]^{r}$ defined as $i\mapsto(\underbrace{1,\dots,1}_{i},0,\dots,0)$ and the diagonal map $X^{\simplex{}{1}}\to(X^{\simplex{}{1}})^{r}$, we obtain a differential form $\omega_{1}\times\dots\times\omega_{r}$ on $X^{\simplex{}{1}}\times\simplex{}{r}$. Then we obtain a differential form on path simplicial set $X^{\simplex{}{1}}$ as a fiberwise integration of $\omega_{1}\times\dots\times\omega_{r}$ along the projection $X^{\simplex{}{1}}\times\simplex{}{r}\to X^{\simplex{}{1}}$. We call it the \Def{iterated integral of $\omega_{1},\dots,\omega_{r}$} and denote it as $\int\omega_{1}\cdots\omega_{r}$. It is precisely an analogy of Chen's iterated integral.
\begin{center}
\begin{tikzpicture}
	\node (03) at (0*4*1, 3*1*1) {$(X^{\simplex{}{1}})^{r}\times\simplex{}{1}^{r}$};
	\node (13) at (1*4*1, 3*1*1) {$(\simplex{}{1}\times X^{\simplex{}{1}})^{r}$};
	\node (02) at (0*4*1, 2*1*1) {$X^{\simplex{}{1}}\times\simplex{}{r}$};
	\node (12) at (1*4*1, 2*1*1) {$X^{r}$};
	\node (22) at (2*4*1, 2*1*1) {$X^{r}$};
	\node (32) at (3*4*1, 2*1*1) {$X$};
	\node (01) at (0*4*1, 1*1*1) {$X^{\simplex{}{1}}\times\simplex{}{r}$};
	\node (11) at (1*4*1, 1*1*1) {$\dR{\simplex{}{-}}{\bullet}{,\Z\langle\dq\rangle}$};
	\node (21) at (2*4*1, 1*1*1) {$\dR{\simplex{}{-}}{\bullet}{,\Z\langle\dq\rangle}^{r}$};
	\node (31) at (3*4*1, 1*1*1) {$\dR{\simplex{}{-}}{\bullet}{,\Z\langle\dq\rangle}$};
	\node (00) at (0*4*1, 0*1*1) {$X^{\simplex{}{1}}$};
	\node (10) at (1*4*1, 0*1*1) {$\dR{\simplex{}{-}}{\bullet}{,\Z\langle\dq\rangle}$};
	\draw[transform canvas={yshift= 1pt}] (03) -- (13);
	\draw[transform canvas={yshift=-1pt}] (03) -- (13);
	\draw[->] (02) --node[left] {$\scriptstyle \text{diagonal}\times\iota_{r}$} (03);
	\draw[->] (13) --node[right] {$\scriptstyle (\mathrm{ev})^{r}$} (12);
	\draw[->] (02) --node[above] {$\scriptstyle \phi_{r}$} (12);
	\draw[transform canvas={yshift= 1pt}] (12) -- (22);
	\draw[transform canvas={yshift=-1pt}] (12) -- (22);
	\path[draw, ->] (22) --node[above] {$\scriptstyle \proj{i}$} (32);
	\draw[transform canvas={xshift= 1pt}] (02) -- (01);
	\draw[transform canvas={xshift=-1pt}] (02) -- (01);
	\draw[->] (12) --node[left] {$\scriptstyle \proj{1}^{\ast}\omega_{1}\wedge\dots\wedge\proj{r}^{\ast}\omega_{r}$} (11);
	\draw[->] (22) --node[right] {$\scriptstyle \sqcap_{i}\proj{i}^{\ast}\omega_{i}$} (21);
	\draw[->] (32) --node[right] {$\scriptstyle \omega_{i}$} (31);
	\draw[->] (01) --node[below] {$\scriptstyle \omega_{1}\times\dots\times\omega_{r}$} (11);
	\draw[->] (21) --node[below] {$\scriptstyle \wedge$} (11);
	\draw[->] (21) --node[below] {$\scriptstyle \proj{i}$} (31);
	\draw[->] (00) --node[above] {$\scriptstyle \int\omega_{1}\dots\omega_{r}$} (10);
\end{tikzpicture}
\end{center}
\begin{prop}\label{differential of iterated integral}
For each homogeneous differential forms $\omega_{1},\dots,\omega_{r}$ on $X$,
\begin{align*}
	d(\int\omega_{1}\dots\omega_{r})
		=&
	\sum_{i=1}^{r}(-1)^{|\omega_{1}|+\dots+|\omega_{i-1}|+r}(\int\omega_{1}\dots d\omega_{i}\dots\omega_{r})
		+
	\sum_{i=1}^{r-1}(-1)^{r-1-i}(\int\omega_{1}\dots(\omega_{i}\wedge\omega_{i+1})\dots\omega_{r})\\
		&+
	(-1)^{(r-1)(|\omega_{1}|-1)}\mathrm{E}_{1}^{\ast}\omega_{1}\wedge(\int\omega_{2}\dots\omega_{r})
		-
	(\int\omega_{1}\dots\omega_{r-1})\wedge\mathrm{E}_{0}^{\ast}\omega_{r}.
\end{align*}
	holds where $\mathrm{E}_{\ep}\colon[\simplex{}{1},X]\to X$ is obtained as a composition $[\simplex{}{1},X]\to[\Delta\{\ep\},X]\cong\simplex{}{0}\times[\simplex{}{0},X]\xrightarrow{\mathrm{ev}}X$ for each $\ep=0,1$.
\end{prop}
\begin{proof}
	From Stokes's theorem \ref{Stokes}, the following holds:
	\begin{align*}
		(-1)^{r}d(\int\omega_{1}\dots\omega_{r})
			=&
		\sum_{i=1}^{r}(-1)^{|\omega_{1}|+\dots+|\omega_{i-1}|}\fint{\proj{[\simplex{}{1},X]}}\phi_{r}^{\ast}(\proj{1}^{\ast}\omega_{1}\wedge\dots\wedge\proj{i}^{\ast}d\omega_{i}\wedge\dots\wedge\proj{r}^{\ast}\omega_{r})\\
			&+
		\sum_{i=0}^{r}(-1)^{i+1}\fint{\proj{[\simplex{}{1},X]}}(\id{}\times\delta_{i})^{\ast}\phi_{r}^{\ast}(\proj{1}^{\ast}\omega_{1}\wedge\dots\wedge\proj{r}^{\ast}\omega_{r})
	\end{align*}
	\begin{center}
	\begin{tikzpicture}[auto]
		\node (11) at (0, 4) {$X^{\simplex{}{1}}\times\simplex{}{r-1}$};
		\node (13) at (8, 4) {$\simplex{}{r-1}\times X^{\simplex{}{1}}$};
		\node (21) at (0, 3) {$X^{\simplex{}{1}}\times\simplex{}{r}$};
		\node (22) at (4, 3) {$X^{\simplex{}{1}}\times\simplex{}{r}$};
		\node (23) at (8, 3) {$\simplex{}{r}\times X^{\simplex{}{1}}$};
		\node (31) at (0, 2) {$(X^{\simplex{}{1}})^{r}\times\simplex{}{1}^{r}$};
		\node (32) at (4, 2) {$X^{\simplex{}{1}}\times\simplex{}{1}^{r}$};
		\node (33) at (8, 2) {$\simplex{}{1}^{r}\times X^{\simplex{}{1}}$};
		\node (41) at (0, 1) {$(\simplex{}{1}\times X^{\simplex{}{1}})^{r}$};
		\node (43) at (8, 1) {$\simplex{}{1}\times X^{\simplex{}{1}}$};
		\node (51) at (0, 0) {$X^{r}$};
		\node (53) at (8, 0) {$X$};
		\path[draw, ->] (11) --node[swap] {$\scriptstyle \id{}\times\delta_{i}$} (21);
		\path[draw, ->] (21) --node[swap] {$\scriptstyle \mathrm{diagonal}\times\iota_{r}$} (31);
		\path[draw, transform canvas={xshift=1pt}] (31) -- (41);
		\path[draw, transform canvas={xshift=-1pt}] (31) --node[swap] {\rotatebox{90}{$\scriptstyle \sim$}} (41);
		\path[draw, ->] (41) --node[swap] {$\scriptstyle \mathrm{ev}$} (51);
		\path[draw, ->] (21) -- (-2, 3) --node[swap] {$\scriptstyle \phi_{r}$} (-2, 0) -- (51);
		\path[draw, ->] (22) --node[swap] {$\scriptstyle \id{}\times\iota_{r}$}  (32);
		\path[draw, ->] (13) --node {$\scriptstyle \delta_{i}\times\id{}$} (23);
		\path[draw, ->] (23) --node {$\scriptstyle \iota_{r}\times\id{}$}  (33);
		\path[draw, ->] (33) --node {$\scriptstyle \proj{j}\times\id{}$} (43);
		\path[draw, ->] (43) --node {$\scriptstyle \mathrm{ev}$} (53);
		\path[draw, transform canvas={yshift=1pt}] (21) -- (22);
		\path[draw, transform canvas={yshift=-1pt}] (21) -- (22);
		\path[draw, transform canvas={yshift=1pt}] (22) --node {$\scriptstyle \sim$} (23);
		\path[draw, transform canvas={yshift=-1pt}] (22) -- (23);
		\path[draw, ->] (31) --node {$\scriptstyle \proj{j}\times\id{}$} (32);
		\path[draw, transform canvas={yshift=1pt}] (32) --node {$\scriptstyle \sim$} (33);
		\path[draw, transform canvas={yshift=-1pt}] (32) -- (33);
		\path[draw, transform canvas={yshift=1pt}] (11) --node {$\scriptstyle \sim$} (13);
		\path[draw, transform canvas={yshift=-1pt}] (11) -- (13);
		\path[draw, ->] (41) --node {$\scriptstyle \proj{j}$} (43);
		\path[draw, ->] (51) --node {$\scriptstyle \proj{j}$} (53);
	\end{tikzpicture}
	\end{center}
	For each pair of $i=0,\dots,r$ and $j=1,\dots,r$, respectively, the following holds:
	\begin{align*}
		\proj{j}\iota_{r}\delta_{i}
			=
		\begin{cases}
			\text{constant $1$}&((i,j)=(0,1))\\
			\proj{j-1}\iota_{r-1}&((i,j)\neq(0,1)\text{ and }i<j)\\
			\proj{j}\iota_{r-1}&((i,j)\neq(r,r)\text{ and }i\geq j)\\
			\text{constant $0$}&((i,j)=(r,r))
		\end{cases}.
	\end{align*}
	In addition, the following diagram is commutative:
	\begin{center}
	\begin{tikzpicture}[auto]
		\node (11) at (0, 2) {$\simplex{}{r-1}\times X^{\simplex{}{1}}$};
		\node (12) at (4, 2) {$\simplex{}{0}\times X^{\simplex{}{1}}$};
		\node (13) at (8, 2) {$\simplex{}{1}\times X^{\simplex{}{1}}$};
		\node (21) at (0, 1) {$\simplex{}{r-1}\times X^{\simplex{}{1}}$};
		\node (22) at (4, 1) {$\simplex{}{0}\times X^{\simplex{}{1}}$};
		\node (23) at (8, 1) {$X$};
		\node (31) at (0, 0) {$X^{\simplex{}{1}}$};
		\node (32) at (4, 0) {$X^{\simplex{}{0}}$};
		\node (33) at (8, 0) {$X$};
		\path[draw, ->] (11) -- (12);
		\path[draw, ->] (31) --node[swap] {$\scriptstyle [\delta_{\ep},X]$} (32);
		\path[draw, double, double distance=2pt] (11) -- (21);
		\path[draw, ->] (21) --node[swap] {$\scriptstyle \proj{[\simplex{}{1},X]}$} (31);
		\path[draw, ->] (22) --node[swap] {$\scriptstyle \proj{[\simplex{}{1},X]}$} (32);
		\path[draw, ->] (21) -- (22);
		\path[draw, ->] (12) --node {$\scriptstyle \delta_{\ep}\times\id{}$} (13);
		\path[draw, ->] (12) --node[swap] {$\scriptstyle \simplex{}{0}\times[\delta_{\ep},X]$} (22);
		\path[draw, ->] (13) --node {$\scriptstyle \mathrm{ev}$} (23);
		\path[draw, ->] (22) --node[swap] {$\scriptstyle \mathrm{ev}$} (23);
		\path[draw, ->] (32) -- (33);
		\path[draw, double, double distance=2pt] (23) -- (33);
	\end{tikzpicture}.
	\end{center}
	Therefore
	\begin{align*}
		&\sum_{i=0}^{r}(-1)^{i+1}\fint{\proj{[\simplex{}{1},X]}}(\id{}\times\delta_{i})^{\ast}\phi_{r}^{\ast}(\proj{1}^{\ast}\omega_{1}\wedge\dots\wedge\proj{r}^{\ast}\omega_{r})\\
			=&
		-\fint{\proj{[\simplex{}{1},X]}}\proj{[\simplex{}{1},X]}^{\ast}\mathrm{E}_{1}^{\ast}\omega_{1}\wedge\phi_{r-1}^{\ast}(\proj{1}^{\ast}\omega_{2}\wedge\dots\wedge\proj{r-1}^{\ast}\omega_{r})\\
			&+
		\sum_{i=1}^{r-1}(-1)^{i+1}\fint{\proj{[\simplex{}{1},X]}}\phi_{r-1}^{\ast}(\proj{1}^{\ast}\omega_{1}\wedge\dots\wedge\proj{i}^{\ast}(\omega_{i}\wedge\omega_{i+1})\wedge\dots\wedge\proj{r}^{\ast}\omega_{r})\\
			&+
		(-1)^{r+1}\fint{\proj{[\simplex{}{1},X]}}(\phi_{r-1}^{\ast}(\proj{1}^{\ast}\omega_{1}\wedge\dots\wedge\proj{r-1}^{\ast}\omega_{r-1})\wedge\proj{[\simplex{}{1},X]}^{\ast}\mathrm{E}_{0}^{\ast}\omega_{r})\\
			=&
		\sum_{i=1}^{r-1}(-1)^{i+1}(\int\omega_{1}\dots(\omega_{i}\wedge\omega_{i+1})\dots\omega_{r})\\
			&+
		(-1)^{|\omega_{1}|(|\omega_{2}|+\dots+|\omega_{r}|)+1}(\int\omega_{2}\dots\omega_{r})\wedge\mathrm{E}_{1}^{\ast}\omega_{1}
			-
		(-1)^{r}(\int\omega_{1}\dots\omega_{r-1})\wedge\mathrm{E}_{0}^{\ast}\omega_{r}\\
			=&
		\sum_{i=1}^{r-1}(-1)^{i+1}(\int\omega_{1}\dots(\omega_{i}\wedge\omega_{i+1})\dots\omega_{r})
			+
		(-1)^{1-|\omega_{1}|(r-1)}\mathrm{E}_{1}^{\ast}\omega_{1}\wedge(\int\omega_{2}\dots\omega_{r})
			-
		(-1)^{r}(\int\omega_{1}\dots\omega_{r-1})\wedge\mathrm{E}_{0}^{\ast}\omega_{r}
	\end{align*}
	holds.
\end{proof}
\begin{prop}\label{wedge of iterated integral}
	For each homogeneous differential forms $\omega_{1},\dots,\omega_{p+q}$ on $X$,
	\begin{align*}
		(\int\omega_{1}\dots\omega_{p})\wedge(\int\omega_{p+1}\dots\omega_{p+q})
			=&
		\sum_{\sigma\in\mathsf{Sh}(p,q)}\ep(\sigma;|\omega_{1}|-1,\dots,|\omega_{p+q}|-1)(\int\omega_{\sigma(1)}\dots\omega_{\sigma(p+q)})
	\end{align*}
	holds.
\end{prop}
\begin{proof}
	This can be shown in the same way as for iterated integral of differential forms on manifolds.
\end{proof}
We describe the domain of iterated integral in the next subsection.

\subsection{iterated integral and homotopy pullback}\label{property sii}
Let $\mathcal{A}$ be a commutative cochain algebra. We define a chain complex $\bm{B}\mathcal{A}$ as follows:
\begin{align*}
	\mathbf{B}\mathcal{A}
		&\coloneqq
	\bigoplus_{r=0}^{\infty}
		\mathcal{A}\otimes(\overline{\mathcal{A}}[1])^{\otimes r}\otimes\mathcal{A},\\
	f[g_{1}|\dots|g_{r}]h
		&\coloneqq
	f\otimes g_{1}[1]\otimes\dots\otimes g_{r}[1]\otimes h,\\
	d
		&\coloneqq
	d_{1}+d_{2}\\
	d_{1}(f[g_{1}|\dots|g_{r}]h)
		&\coloneqq
	df[g_{1}|\dots|g_{r}]h\\
		&-
	\sum_{i=1}^{r}(-1)^{|f|+\nu_{i-1}}f[g_{1}|\dots|dg_{i}|\dots|g_{r}]h\\
		&+
	(-1)^{|f|+\nu_{r}}f[g_{1}|\dots|g_{r}]dh,\\
	d_{2}(f[g_{1}|\dots|g_{r}]h)
		&\coloneqq
	(-1)^{|f|}f\wedge g_{1}[g_{2}|\dots|g_{r}]h\\
		&+
	\sum_{i=1}^{r-1}(-1)^{|f|+\nu_{i}}f[g_{1}|\dots|g_{i}\wedge g_{i+1}|\dots|g_{r}]h\\
		&+
	(-1)^{|f|+\nu_{r-1}+1}f[g_{1}|\dots|g_{r-1}]g_{r}\wedge h,\\
	\nu_{i}
		&\coloneqq
	|g_{1}|+\dots+|g_{i}|-i.
\end{align*}
Then we obtain a canonical surjective quasi-isomorphism $\bm{B}\mathcal{A}\to\mathcal{A}$. In addition, we can define a shuffle product as
\begin{align*}
	(f[g_{1}|\dots|g_{p}]h)\wedge(f'[g_{p+1}|\dots|g_{p+q}]h')
		&\coloneqq
	\sum_{\sigma\in\mathsf{Sh}(p,q)}(\pm)\ep(\sigma,|g_{1}|-1,\dots,|g_{p+q}|-1)f\wedge f'[g_{\sigma(1)}|\dots|g_{\sigma(p+q)}]h\wedge h',\\
	(\pm)
		&\coloneqq
	(-1)^{|h|\cdot|f'|+(|g_{1}|+\dots+|g_{p}|-p)|f'|+|h|(|g_{p+1}|+\dots+|g_{p+q}|-q)}.
\end{align*}
The product gives a commutative cochain algebra $\bm{B}\mathcal{A}$ and a ring morphism $\bm{B}\mathcal{A}\to\mathcal{A}$. It gives a factorization
\begin{equation*}
	\mathcal{A}\otimes\mathcal{A}
		\to
	\mathbf{B}\mathcal{A}
		\to
	\mathcal{A}
\end{equation*}
of the diagonal map $\mathcal{A}\otimes\mathcal{A}\to\mathcal{A}$.

Let $X$ be a simlicial set. Then, the diagonal map $X\to X\times X$ can be factorized into
\begin{equation*}
	X
		\to
	X^{\simplex{}{1}}
		\to
	X\times X.
\end{equation*}
We obtain the following commutative diagram by considering de Rham algebras:
\begin{center}
\begin{tikzpicture}
	\node (01) at (0*3*1.25, 1*1.25) {$\mathcal{A}(X)$};
	\node (00) at (0*3*1.25, 0*1.25) {$\mathcal{A}(X^{\simplex{}{1}})$};
	\node (10) at (1*3*1.25, 0*1.25) {$\mathcal{A}(X\times X)$};
	\node (21) at (2*3*1.25, 1*1.25) {$\mathbf{B}\mathcal{A}(X)$};
	\node (20) at (2*3*1.25, 0*1.25) {$\mathcal{A}(X)\otimes\mathcal{A}(X)$};
	\path[draw, ->>] (00) --node[right] {$\scriptstyle \sim$} (01);
	\path[draw, ->] (10) -- (00);
	\path[draw, ->] (20) --node[above] {$\scriptstyle \sim$} (10);
	\path[draw, ->>] (21) --node[above] {$\scriptstyle \sim$} (01);
	\path[draw, ->] (20) -- (21);
\end{tikzpicture}
\end{center}
\begin{prop}
	The iterated integral on a simplicial set gives a morphism of commutative cochain algebra
	\begin{align*}
		\mathbb{I}
			\colon
		\mathbf{B}\mathcal{A}(X)
			&\to
		\mathcal{A}(X^{\simplex{}{1}})\\
		\omega_{0}[\omega_{1}|\dots|\omega_{r}]\omega_{r+1}
			&\mapsto
		(-1)^{\sum_{i=1}^{r}(r-i)(|\omega_{i}|-1)}
		(\mathsf{E}_{1}^{\ast}\omega_{0})
			\wedge
		(\int\omega_{1}\dots\omega_{r})
			\wedge
		(\mathsf{E}_{0}^{\ast}\omega_{r+1})
	\end{align*}
	such that the following diagram is commutative:
	\begin{center}
	\begin{tikzpicture}
		\node (02) at (0*3*1.25, 2*1.25) {$\mathcal{A}(X)$};
		\node (01) at (0*3*1.25, 1*1.25) {$\mathcal{A}(X^{\simplex{}{1}})$};
		\node (00) at (0*3*1.25, 0*1.25) {$\mathcal{A}(X)\otimes\mathcal{A}(X)$};
		\node (12) at (1*3*1.25, 2*1.25) {$\mathcal{A}(X)$};
		\node (11) at (1*3*1.25, 1*1.25) {$\mathbf{B}\mathcal{A}(X)$};
		\node (10) at (1*3*1.25, 0*1.25) {$\mathcal{A}(X)\otimes\mathcal{A}(X)$};
		\path[draw, transform canvas={yshift= 1pt}] (12) -- (02);
		\path[draw, transform canvas={yshift=-1pt}] (12) -- (02);
		\path[draw, ->] (11) --node[above] {$\scriptstyle \int$} (01);
		\path[draw, ->] (10) --node[above] {$\scriptstyle \sim$} (00);
		\path[draw, ->] (10) -- (11);
		\path[draw, ->>] (11) --node[right] {$\scriptstyle \sim$} (12);
		\path[draw, ->] (00) -- (01);
		\draw[->>] (01) --node[right] {$\scriptstyle \sim$} (02);
	\end{tikzpicture}.
	\end{center}
\end{prop}
\begin{proof}
	From proposition \ref{differential of iterated integral}, the following hold:
	\begin{align*}
		d\mathbb{I}(\omega_{0}[\omega_{1}|\dots|\omega_{r}]\omega_{r+1})
			&=
		(-1)^{\sum_{i=1}^{r}(r-i)(|\omega_{i}|-1)}d((\mathsf{E}_{1}^{\ast}\omega_{0})\wedge(\int\omega_{1}\dots\omega_{r})\wedge(\mathsf{E}_{0}^{\ast}\omega_{r+1}))\\
			&=
		(-1)^{\sum_{i=1}^{r}(r-i)(|\omega_{i}|-1)}(\mathsf{E}_{1}^{\ast}d\omega_{0})\wedge(\int\omega_{1}\dots\omega_{r})\wedge(\mathsf{E}_{0}^{\ast}\omega_{r+1})\\
			&+
		(-1)^{\sum_{i=1}^{r}(r-i)(|\omega_{i}|-1)+|\omega_{0}|}(\mathsf{E}_{1}^{\ast}\omega_{0})\wedge d(\int\omega_{1}\dots\omega_{r})\wedge(\mathsf{E}_{0}^{\ast}\omega_{r+1})\\
			&+
		(-1)^{\sum_{i=1}^{r}(r-i)(|\omega_{i}|-1)+|\omega_{0}|+\sum_{i=1}^{r}(|\omega_{i}|-1)}(\mathsf{E}_{1}^{\ast}\omega_{0})\wedge(\int\omega_{1}\dots\omega_{r})\wedge(\mathsf{E}_{0}^{\ast}d\omega_{r+1})\\
			&=
		(-1)^{\sum_{i=1}^{r}(r-i)(|\omega_{i}|-1)}(\mathsf{E}_{1}^{\ast}d\omega_{0})\wedge(\int\omega_{1}\dots\omega_{r})\wedge(\mathsf{E}_{0}^{\ast}\omega_{r+1})\\
			&+
		(-1)^{\sum_{i=1}^{r}(r-i)(|\omega_{i}|-1)+|\omega_{0}|}(\mathsf{E}_{1}^{\ast}\omega_{0})\wedge 
		(\sum_{i=1}^{r}(-1)^{|\omega_{1}|+\dots+|\omega_{i-1}|+r}(\int\omega_{1}\dots d\omega_{i}\dots\omega_{r}))
		\wedge(\mathsf{E}_{0}^{\ast}\omega_{r+1})\\
			&+
		(-1)^{\sum_{i=1}^{r}(r-i)(|\omega_{i}|-1)+|\omega_{0}|}(\mathsf{E}_{1}^{\ast}\omega_{0})\wedge 
		(\sum_{i=1}^{r-1}(-1)^{r-1-i}(\int\omega_{1}\dots(\omega_{i}\wedge\omega_{i+1})\dots\omega_{r}))
		\wedge(\mathsf{E}_{0}^{\ast}\omega_{r+1})\\
			&+
		(-1)^{\sum_{i=1}^{r}(r-i)(|\omega_{i}|-1)+|\omega_{0}|}(\mathsf{E}_{1}^{\ast}\omega_{0})\wedge 
		((-1)^{(r-1)(|\omega_{1}|-1)}\mathrm{E}_{1}^{\ast}\omega_{1}\wedge(\int\omega_{2}\dots\omega_{r}))
		\wedge(\mathsf{E}_{0}^{\ast}\omega_{r+1})\\
			&-
		(-1)^{\sum_{i=1}^{r}(r-i)(|\omega_{i}|-1)+|\omega_{0}|}(\mathsf{E}_{1}^{\ast}\omega_{0})\wedge 
		((\int\omega_{1}\dots\omega_{r-1})\wedge(\mathrm{E}_{0}^{\ast}\omega_{r}))
		\wedge(\mathsf{E}_{0}^{\ast}\omega_{r+1})\\
			&+
		(-1)^{\sum_{i=1}^{r}(r-i)(|\omega_{i}|-1)+|\omega_{0}|+\sum_{i=1}^{r}(|\omega_{i}|-1)}(\mathsf{E}_{1}^{\ast}\omega_{0})\wedge(\int\omega_{1}\dots\omega_{r})\wedge(\mathsf{E}_{0}^{\ast}d\omega_{r+1})\\
			&=
		\mathbb{I}d(\omega_{0}[\omega_{1}|\dots|\omega_{r}]\omega_{r+1}).
	\end{align*}
	Therefore it is a cochain map. From proposition \ref{wedge of iterated integral}, it is an algebra map. As for the rest, it is clear.
\end{proof}
At this time, we describe that the (simplicial) iterated integral
\begin{equation*}
	\int
		\colon
	\Z\langle\dq\rangle\bigotimes_{\dR{X}{\bullet}{,\Z\langle\dq\rangle}\otimes\dR{X}{\bullet}{,\Z\langle\dq\rangle}}\bm{B}\dR{X}{\bullet}{,\Z\langle\dq\rangle}
		\to
	\dR{\{x\}\prod_{X\times X}X^{\simplex{}{1}}}{\bullet}{,\Z\langle\dq\rangle}
\end{equation*}
is given by some universal property. The diagonal map $X\to X\times X$ and a $0$-simplex $x\in X_{0}$ give a Cartesian diagrams
\begin{center}
\begin{tikzpicture}
	\node (23) at (2*6, 3*1) {$\displaystyle X$};
	\node (22) at (2*6, 2*1) {$\displaystyle X^{\simplex{}{1}}$};
	\node (20) at (2*6, 0*1) {$\displaystyle X\times X$};
	\node (12) at (1*6, 2*1) {$\displaystyle X\prod^{}_{X\times X}X^{\simplex{}{1}}$};
	\node (10) at (1*6, 0*1) {$\displaystyle X$};
	\node (02) at (0*6, 2*1) {$\displaystyle \{x\}\prod^{}_{X\times X}X^{\simplex{}{1}}$};
	\node (00) at (0*6, 0*1) {$\displaystyle \{x\}$};
	\path[draw, right hook->] (23) --node[right] {$\scriptstyle \sim$} (22);
	\path[draw, ->] (22) -- (20);
	\path[draw, ->] (10) -- (20);
	\path[draw, ->] (12) -- (10);
	\path[draw, ->] (12) -- (22);
	\path[draw, ->] (00) -- (10);
	\path[draw, ->] (02) -- (00);
	\path[draw, ->] (02) -- (12);
\end{tikzpicture}.
\end{center}
By considering the de Rham algebras of them, we obtain a commutative diagram
\begin{center}
\begin{tikzpicture}
	\node (23) at (2*6, 3*1) {$\displaystyle \dR{X}{\bullet}{,\Z\langle\dq\rangle}$};
	\node (22) at (2*6, 2*1) {$\displaystyle \dR{X^{\simplex{}{1}}}{\bullet}{,\Z\langle\dq\rangle}$};
	\node (20) at (2*6, 0*1) {$\displaystyle \dR{X\times X}{\bullet}{,\Z\langle\dq\rangle}$};
	\node (12) at (1*6, 2*1) {$\displaystyle \dR{X\prod^{}_{X\times X}X^{\simplex{}{1}}}{\bullet}{,\Z\langle\dq\rangle}$};
	\node (10) at (1*6, 0*1) {$\displaystyle \dR{X}{\bullet}{,\Z\langle\dq\rangle}$};
	\node (02) at (0*6, 2*1) {$\displaystyle \dR{\{x\}\prod^{}_{X\times X}X^{\simplex{}{1}}}{\bullet}{,\Z\langle\dq\rangle}$};
	\node (00) at (0*6, 0*1) {$\displaystyle \Z\langle\dq\rangle$};
	\path[draw, <<-] (23) --node[right] {$\scriptstyle \sim$} (22);
	\path[draw, <-] (22) -- (20);
	\path[draw, <-] (10) -- (20);
	\path[draw, <-] (12) -- (10);
	\path[draw, <-] (12) -- (22);
	\path[draw, <-] (00) -- (10);
	\path[draw, <-] (02) -- (00);
	\path[draw, <-] (02) -- (12);
%	\path[draw] (0*2*1+0.3, 0.5*1.5*1) -- (0.5*1*1, 0.5*1.5*1) -- (0.5*1*1, 1*1.5*1-0.3);
\end{tikzpicture}.
\end{center}
Then, we obtain the following diagram:
\begin{center}
\begin{tikzpicture}
	\node (26) at (2*2.5, 6*1.5) {$\displaystyle \dR{X}{\bullet}{,\Z\langle\dq\rangle}$};
	\node (25) at (2*2.5, 5*1.5) {$\displaystyle \dR{X^{\simplex{}{1}}}{\bullet}{,\Z\langle\dq\rangle}$};
	\node (22) at (2*2.5, 2*1.5) {$\displaystyle \dR{X\times X}{\bullet}{,\Z\langle\dq\rangle}$};
	\draw[<<-] (26) --node {$\scriptstyle \sim$} (25);
	\draw[<-] (25) -- (22);
	\node (56) at (5*2.5, 6*1.5) {$\displaystyle \dR{X}{\bullet}{,\Z\langle\dq\rangle}$};
	\node (55) at (5*2.5, 5*1.5) {$\displaystyle \mathbf{B}\dR{X}{\bullet}{,\Z\langle\dq\rangle}$};
	\node (52) at (5*2.5, 2*1.5) {$\displaystyle \dR{X}{\bullet}{,\Z\langle\dq\rangle}\otimes\dR{X}{\bullet}{,\Z\langle\dq\rangle}$};
	\draw[<<-] (56) --node {$\scriptstyle \sim$} (55);
	\draw[<-] (55) -- (52);
	\draw[transform canvas={yshift= 1pt}] (26) -- (56);
	\draw[transform canvas={yshift=-1pt}] (26) -- (56);
	\draw[<-] (22) --node[above] {$\scriptstyle \sim$} (52);
	\draw[<-] (25) --node[above] {$\scriptstyle \sim$} (55);
	\path (25) --node[below] {$\scriptstyle \mathbb{I}$} (55);
	\node (11) at (1*2.5, 1*1.5) {$\displaystyle \dR{X}{\bullet}{,\Z\langle\dq\rangle}$};
	\node (00) at (0*2.5, 0*1.5) {$\displaystyle \dR{\{x\}}{\bullet}{,\Z\langle\dq\rangle}$};
	\draw[->] (22) -- (11);
	\draw[->] (11) -- (00);
	\node (41) at (4*2.5, 1*1.5) {$\displaystyle \dR{X}{\bullet}{,\Z\langle\dq\rangle}$};
	\node (30) at (3*2.5, 0*1.5) {$\displaystyle \Z\langle\dq\rangle$};
	\draw[->] (52) -- (41);
	\draw[->] (41) -- (30);
	\path[draw, transform canvas={yshift= 1pt}] (11) -- (41);
	\path[draw, transform canvas={yshift=-1pt}] (11) -- (41);
	\path[draw, transform canvas={yshift= 1pt}] (00) -- (30);
	\path[draw, transform canvas={yshift=-1pt}] (00) -- (30);
	\node (14) at (1*2.5, 4*1.5) {$\displaystyle \dR{X\prod^{}_{X\times X}X^{\simplex{}{1}}}{\bullet}{,\Z\langle\dq\rangle}$};
	\node (03) at (0*2.5, 3*1.5) {$\displaystyle \dR{\{x\}\prod^{}_{X\times X}X^{\simplex{}{1}}}{\bullet}{,\Z\langle\dq\rangle}$};
	\draw[<-] (14) -- (11);
	\draw[<-] (03) -- (00);
	\draw[->] (25) -- (14);
	\draw[->] (14) -- (03);
	\node (44) at (4*2.5, 4*1.5) {$\displaystyle \dR{X}{\bullet}{,\Z\langle\dq\rangle}\bigotimes^{}_{\dR{X}{\bullet}{,\Z\langle\dq\rangle}\otimes\dR{X}{\bullet}{,\Z\langle\dq\rangle}}\bm{B}\dR{X}{\bullet}{,\Z\langle\dq\rangle}$};
	\node (33) at (3*2.5, 3*1.5) {$\displaystyle \Z\langle\dq\rangle\bigotimes^{}_{\dR{X}{\bullet}{,\Z\langle\dq\rangle}\otimes\dR{X}{\bullet}{,\Z\langle\dq\rangle}}\bm{B}\dR{X}{\bullet}{,\Z\langle\dq\rangle}$};
	\draw[<-] (44) -- (41);
	\draw[<-] (33) -- (30);
	\draw[->] (55) -- (44);
	\draw[->] (44) -- (33);
	\draw[<-, dashed] (14) --node[below] {$\scriptstyle \int$} (44);
	\draw[<-, dashed] (03) --node[below] {$\scriptstyle \int$} (33);
\end{tikzpicture}.
\end{center}
The induced map is exactly the iterated integral (on a simplicial set).

Let $\K$ be a field of characteristic $0$. For each $n\geq0$, there exists a natural morphism $\Z\langle\dq,x_{1},\dots,x_{n}\rangle\to\K[x_{1},\dots,x_{n}]$ determined by $\dq^{[N]}\mapsto\frac{1}{N!}$ and $x_{i}^{[N]}\mapsto\frac{x_{i}^{N}}{N!}$. It gives a simplicial commutative cochain algebra $\dR{\simplex{}{-}}{\bullet}{,\K}$ defined by
\begin{equation*}
	\dR{\simplex{}{n}}{\bullet}{,\K}
		\coloneqq
	\dR{\simplex{}{n}}{\bullet}{,\Z\langle\dq\rangle}_{\Z\langle\dq,x_{1},\dots,x_{n}\rangle}\K[x_{1},\dots,x_{n}].
\end{equation*}
It gives a commutative cochain algebra $\dR{X}{\bullet}{,\K}$ for any simplicial set. The fiberwise integration (on simplicial sets) can be defined in the same way when the coefficient ring is $\K$. Thus, we can define morphisms
\begin{align*}
	\bm{B}\dR{X}{\bullet}{,\K}
		&\to
	\dR{X^{\simplex{}{1}}}{\bullet}{,\K},\\
	\dR{X}{\bullet}{,\K}\bigotimes^{}_{\dR{X}{\bullet}{,\K}\otimes\dR{X}{\bullet}{,\K}}\bm{B}\dR{X}{\bullet}{,\K}
		&\to
	\dR{X\prod_{X\times X}X^{\simplex{}{1}}}{\bullet}{,\K},\\
	\K\bigotimes^{}_{\dR{X}{\bullet}{,\K}\otimes\dR{X}{\bullet}{,\K}}\bm{B}\dR{X}{\bullet}{,\K}
		&\to
	\dR{\{x\}\prod_{X\times X}X^{\simplex{}{1}}}{\bullet}{,\K}.
\end{align*}
\begin{thm}\label{iterated integral and homotopy pullback}
	The simplicial ($\K$ coefficient) iterated integral (that is the above maps) gives a representation of the map which induced by the universality of homotopy pushout.
\end{thm}
\begin{proof}
	By assumption, the algebra
	\begin{equation*}
		\dR{X}{\bullet}{,\K}\bigotimes^{}_{\dR{X}{\bullet}{,\K}\otimes\dR{X}{\bullet}{,\K}}\bm{B}\dR{X}{\bullet}{,\K}
	\end{equation*}
	coincide with the derived tensor $\dR{X}{\bullet}{,\K}\bigotimes^{\bm{L}}_{\dR{X}{\bullet}{,\K}\otimes\dR{X}{\bullet}{,\K}}\dR{X}{\bullet}{,\K}$, and the algebra
	\begin{equation*}
		\K\bigotimes^{}_{\dR{X}{\bullet}{,\K}\otimes\dR{X}{\bullet}{,\K}}\bm{B}\dR{X}{\bullet}{,\K}
	\end{equation*}
	coincide with the derived tensor $\K\bigotimes^{\bm{L}}_{\dR{X}{\bullet}{,\K}\otimes\dR{X}{\bullet}{,\K}}\dR{X}{\bullet}{,\K}$. The rest immediately follow from the definition of homotopy pushout (of commutative cochain algebras over a field whose character is $0$).
\end{proof}
\begin{obs}[The displacement of two types of ``loop spaces'']\label{iterated integral and homotopy pullback}
	In the opposite category $\mathsf{CDGA}(\R)\opposite$, the derived tensor product is the homotopy pullback. Therefore, roughly speaking, \bf the (simplicial) iterated integral is a map comparing two ``homotopy pullbacks'', one is a ``spacial loop space'' and another one is an ``algebraic loop space''\rm . (The following is an informal diagram in the opposite category $\mathsf{CDGA}\opposite$.)
	\begin{center}
	\begin{tikzpicture}
		\node (13) at (1*3, 3*1.0) {$\displaystyle \mathcal{A}(X)$};
		\node (11) at (1*3, 1*1.0) {$\displaystyle \mathcal{A}(X\times X)$};
		\node (00) at (0*3, 0*1.0) {$\displaystyle \mathcal{A}(X)$};
		\node (02) at (0*3, 2*1.0) {$\displaystyle \mathcal{A}(X\prod^{\bm{R}}_{X\times X}X)$};
		\draw[->] (13) -- (11);
		\draw[->] (00) -- (11);
		\draw[->] (02) -- (00);
		\draw[->] (02) -- (13);
		\node (20) at (2*3, 0*1.0) {$\displaystyle \mathcal{A}(X)$};
		\node (31) at (3*3, 1*1.0) {$\displaystyle \mathcal{A}(X)\otimes\mathcal{A}(X)$};
		\node (33) at (3*3, 3*1.0) {$\displaystyle \mathcal{A}(X)$};
		\draw[double, double distance=2pt] (13) -- (33);
		\draw[->] (11) --node[above, pos=0.3] {$\scriptstyle \simeq$} (31);
		\draw[double, double distance=2pt] (00) -- (20);
		\draw[->] (33) -- (31);
		\draw[->] (20) -- (31);
		\node (22) at (2*3, 2*1.0) {$\displaystyle \mathcal{A}(X)\prod^{\bm{R}}_{\mathcal{A}(X)\otimes\mathcal{A}(X)}\mathcal{A}(X)$};
		\draw[white, line width=2pt] (22) -- (20);
		\draw[->] (22) -- (20);
		\draw[->] (22) -- (33);
		\draw[white, line width=2pt] (02) -- (22);
		\draw[->, dashed] (02) -- (22);
	\end{tikzpicture}
	\end{center}
\end{obs}

\section{comparison theorem}
Roughly speaking, iterated integral is a technique developed by K.T.Chen to construct differential forms on a loop space of ``smooth space''. Chen defined \bf Chen space\rm{} to construct a loop space of smooth manifold which has a ``smooth structure''. On the other hand, Souriau defined \bf diffeological space\rm{} as a kind of generalization of smooth manifolds, and I.Zemmour developed it. It is a well-known fact that there is a ``canonical'' diffeological space of smooth loops of manifold. Roughly speaking, it is the smooth or diffeological loop space of a manifold. Under this wording,
\begin{center}\it
	Chen's iterated integral is a technique to construct differential forms on a smooth loop space. %of a smooth space. 
\end{center}
On the other hand, we already know the simplicial analogy of Chen's iterated integral. In short,
\begin{center}\it
	the simplicial iterated integral is a technique to construct differential forms on a simplicial loop space.% of a simplicial set 
\end{center}
There is some way to construct a simplicial set from a smooth manifold. Hence we obtain a question: are these two kinds of iterated integrals equivalent?

To describe this comparison, we consider two different comparisons. One is a comparison of loop spaces, which constitutes a zig-zag of simplicial maps
\begin{equation*}
	\Omega_{x}S(\mathcal{M})
		=
	\Omega_{x}S^{\mathcal{D}}\mathsf{Cdiff}\mathsf{Dtplg}(\mathcal{M})
		\xleftarrow{\text{(\ref{loop space comparison 1})}}
	\Omega_{x}S^{\mathcal{D}}(\mathcal{M})
		\xleftarrow{\text{(\ref{loop space comparison 2})}}
	S^{\mathcal{D}}(\mathsf{Loop}(\mathcal{M},x))
		\xleftarrow{\text{(\ref{Kihara2})}}
	S^{\mathcal{D}}_{\mathsf{aff}}(\mathsf{Loop}(\mathcal{M},x))
\end{equation*}
between several types of loop spaces. The other one is a comparison of de Rham algebras. To compare several types of de Rham algebras, we construct a zig-zag of natural transformations
\begin{center}
\begin{tikzpicture}
	\node (02) at (0*3, 2*1.2) {$\mathsf{Diff}$};
	\node (12) at (1*3, 2*1.2) {$\mathsf{Diff}$};
	\node (22) at (2*3, 2*1.2) {$\mathsf{Diff}$};
	\node (32) at (3*3, 2*1.2) {$\mathsf{Diff}$};
	\draw[transform canvas={yshift= 1pt}] (02) -- (12) -- (22) -- (32);
	\draw[transform canvas={yshift=-1pt}] (02) -- (12) -- (22) -- (32);
	\node (01) at (0*3, 1*1.2) {$\sSet$};
	\node (11) at (1*3, 1*1.2) {$\sSet$};
	\node (21) at (2*3, 1*1.2) {$\sSet$};
	\draw[transform canvas={yshift= 1pt}] (01) -- (11) -- (21);
	\draw[transform canvas={yshift=-1pt}] (01) -- (11) -- (21);
	\node (00) at (0*3, 0*1.2) {$\mathsf{CAlg}\opposite$};
	\node (10) at (1*3, 0*1.2) {$\mathsf{CAlg}\opposite$};
	\node (20) at (2*3, 0*1.2) {$\mathsf{CAlg}\opposite$};
	\node (30) at (3*3, 0*1.2) {$\mathsf{CAlg}\opposite$};
	\draw[transform canvas={yshift= 1pt}] (00) -- (10) -- (20) -- (30);
	\draw[transform canvas={yshift=-1pt}] (00) -- (10) -- (20) -- (30);
	\draw[->] (02) --node[left] {$\scriptstyle S^{\mathcal{D}}_{\mathsf{aff}}$} (01);
	\path (02) --node[right] (0201R) {} (01);
	\draw[->] (01) --node[left] {$\scriptstyle \dR{-}{\bullet}{,\Z\langle\dq\rangle}$} (00);
	\path (01) --node[right] (0100R) {} (00);
	\draw[->] (12) --node[left] (1211L) {$\scriptstyle S^{\mathcal{D}}_{\mathsf{aff}}$} (11);
	\path (12) --node[right] (1211R) {} (11);
	\draw[->] (11) --node[left] (1110L) {$\scriptstyle \mathcal{A}_{\mathrm{PP}}$} (10);
	\path (11) --node[right] (1110R) {} (10);
	\draw[->] (22) --node[left] (2221L) {$\scriptstyle S^{\mathcal{D}}_{\mathsf{aff}}$} (21);
	\draw[->] (21) --node[left] (2120L) {$\scriptstyle \mathcal{A}_{\mathrm{PS}}$} (20);
	\draw[->] (32) --node[left] (3230L) {$\scriptstyle \mathcal{A}_{\mathsf{dR}}$} (30);
	\draw[double, double distance=2pt, shorten <= 15pt, shorten >=15pt] (0201R) -- (1211L);
	\draw[double, double distance=2pt, ->, -implies, shorten <= 15pt, shorten >=15pt] (0100R) -- (1110L);
	\draw[double, double distance=2pt, shorten <= 15pt, shorten >=15pt] (1211R) -- (2221L);
	\draw[double, double distance=2pt, ->, -implies, shorten <= 15pt, shorten >=15pt] (1110R) -- (2120L);
	\draw[double, double distance=2pt, ->, -implies, shorten <= 15pt, shorten >=15pt] (3230L) --node[above] {$\scriptstyle \alpha$} (21);
\end{tikzpicture}
\end{center}
since assigning de Rham algebra to each diffeological space gives a functor.

\subsection{comparison of loop spaces}\label{review diff}
Let $\mathsf{Dom}$ be a category whose objects are the open subset of the Euclidean space $\R^{n}$ and whose morphisms are smooth maps between these. For a given set $D$, we have two contravariant functors $\mathcal{W}_{D},\mathcal{K}_{D}\colon\mathsf{Dom}\opposite\to\Set$ defined by
\begin{enumerate}
	\item $\mathcal{W}_{D}(D)\coloneqq\mathsf{Map}(U,D)$ and $\mathcal{W}_{D}(\phi)(p)=p\circ\phi$,
	\item $\mathcal{W}_{D}(D)\coloneqq\{q\in\mathcal{W}_{D}(U)|\text{$q$ is locally constant}\}$ and $\mathcal{K}_{D}(\phi)(q)=q\circ\phi$,
\end{enumerate}
for any open subsets $U\subset\R^{n}, V\subset\R^{m}$, any elements $p\in\mathcal{W}_{D}(U)$, $q\in\mathcal{K}_{D}(U)$ and smooth map $\phi\colon V\to U$, where $q\colon U\to D$ is said to be locally constant, if there exists an open covering $\bigcup_{\alpha}V_{\alpha}=U$ of $U$ such that $q|_{V_{\alpha}}$ is constant for all index $\alpha$.
\begin{dfn}[diffeology]
	We call a pair $(D,\mathcal{D})$ a \defn{diffeological space}, if it satisfies the following conditions.
	\begin{enumerate}
		\item $D$ is a set and $\mathcal{D}_{D}\colon\mathsf{Dom}\opposite\to\mathsf{Set}$ is a contravariant functor.
		\item $\mathcal{K}_{D}\subset\mathcal{D}_{D}\subset\mathcal{W}_{D}$ as contravariant functor.
		\item For any open subset $U$ and its open covering $U=\bigcup_{\alpha}U_{\alpha}$, $p\in\mathcal{D}_{D}(U)$ if and only if $p|_{U_{\alpha}}\in\mathcal{D}_{D}(U_{\alpha})$ for all $\alpha$.
	\end{enumerate}
\end{dfn}
\begin{ex}[diffeology of a manofild]
	Let $\mathcal{M}$ be a smooth manifold. We define the canonical diffeology $\mathcal{D}_{\mathcal{M}}$ of $\mathcal{M}$ as
	\begin{equation*}
		\mathcal{D}_{\mathcal{M}}(U)
			\coloneqq
		\{p\colon U\to\mathcal{M}|\text{$p$ is smooth.}\}
	\end{equation*}
	for each open set $U$.
\end{ex}
\begin{ex}[functional diffeology]
	Let $D_{1},D_{2}$ be diffeologiacl spaces. We denote the set $\{\varphi\colon D_{1}\to D_{2}\}$ of smooth maps by $C^{\infty}(D_{1},D_{2})$. We define a diffeology of $C^{\infty}(D_{1},D_{2})$ as
	\begin{equation*}
		\mathcal{D}_{C^{\infty}(D_{1},D_{2})}(U)
			\coloneqq
		\{p\colon U\to C^{\infty}(D_{1},D_{2})|\text{$p^{\vee}$ is smooth.}\}
	\end{equation*}
	for each open set $U$, where $p^{\vee}\colon U\times D_{1}\to D_{2}$ is a map defined by $p^{\vee}(u,x)\coloneqq p(u)(x)$.
	We call the diffeology \defn{functional diffeology}.
\end{ex}
\begin{ex}[continuous diffeology and $\mathcal{D}$-topology]
	Conversely, for any diffeological space $(D,\mathcal{D})$, there is a topological space $\mathsf{Dtplg}(D)$ whose family of open sets is defined by
	\begin{equation*}
		\{O\subset D|p^{-1}(O)\text{ is open for any plot $p\in\mathcal{D}$}\}.
	\end{equation*}
	It gives a left adjoit $\mathsf{Dtplg}$ of $\mathsf{Cdiff}\colon\mathsf{Top}\to\mathsf{Diff}$.
\end{ex}

Let $D$ be a diffeological space. We call a \defn{smooth path} in $D$ any smooth map from $\R$ to $D$, and denote the diffeological space $C^{\infty}(\R,D)$ of all smooth paths in $D$ by $\mathsf{Path}(D)$. We say that a smooth path $\gamma\colon\R\to D$ is \defn{stationary} if there exists a positive real number $\ep>0$ such that $f$ is constant on $(-\infty,\ep)$ and also on $(1-\ep,\infty)$.
We denote the subdiffeological space of stationary paths in $D$ by $\mathsf{stPath}(D)$. Since, for each point $t\in\R$, the evaluation map $\mathsf{ev}_{t}\colon C^{\infty}(\R,D)\to D$ defined by $\gamma\mapsto\gamma(t)$ is smooth, we obtain a ``projection'' $(\mathsf{ev}_{0},\mathsf{ev}_{1})\colon\mathsf{Path}(D)\to D\times D$. For each point $x\in D$, we denote the pullback of $\mathsf{Path}(D)\to D\times D$ along $\{(x,x)\}\to D\times D$ as $\mathsf{Loop}(D,x)$. Similarly, we denote the pullback of $\mathsf{stPath}(D)\to D\times D$ along $\{(x,x)\}\to D\times D$ as $\mathsf{stLoop}(D,x)$.

We define an equivalence relation on $D$ by $x\sim y$ if and only if there is a stationary smooth path $\gamma$ which satisfies $\gamma(0)=x$ and $\gamma(1)=y$. We call the quotient set $D/\sim$ the $0$th \defn{smooth homotopy group} and denote it by $\pi^{\mathcal{D}}_{0}(D)$. The first \defn{smooth homotopy group} $\pi^{\mathcal{D}}_{1}(D,x)$ is defined as $\pi^{\mathcal{D}}_{0}(\mathsf{Loop}(D,x),\mathsf{const}_{x})$. In addition, the $n$th \defn{smooth homotopy group} $\pi^{\mathcal{D}}_{n}(D,x)$ is inductively defined as $\pi^{\mathcal{D}}_{n-1}(\mathsf{Loop}(D,x),\mathsf{const}_{x})$ for each $n\geq2$.

A \defn{pointed diffeological space} is a pair $(D,x)$ of a diffeological space $X$ and its point $x\in D$. This seems to be a special case of the pair of a diffeological space $D$ and its diffeological subspace $\{x\}$, if we consider the point $x$ as a subspace $\{x\}$. For any pairs $(X,U), (Y,V)$ of diffeological space and its subspace, we denote the set of smooth functions $f\colon X\to Y$ satisfying $f(U)\subset V$ with the sub-diffeology from $C^{\infty}(X,Y)$ as $C^{\infty}((X,U),(Y,V))$, and we denote $\pi^{\mathcal{D}}_{0}(C^{\infty}((X,U),(Y,V)))$ as $[(X,U),(Y,V)]$.

\begin{thm}[Christensen-Wu \cite{christensen2015homotopy} Theorem 3.2]\label{Christensen-Wu}
	Let $(D,x)$ be a pointed diffeological space, $n\geq1$ be a positive integer and $\ep\in(0,\frac{1}{2})$ be a positive real number. In addition, we define diffeological spaces $\partial\R^{n}$, $\partial_{\ep}\R^{n}$, $\mathbb{A}^{n}$, $\partial_{\ep}\mathbb{A}^{n}$, $\diffsimplex{n}_{\mathsf{sub}}$ and $\partial_{\ep}\diffsimplex{n}_{\mathsf{sub}}$ respectively as follows:
	\begin{enumerate}
		\item $\partial\R^{n}=\{(s_{1},\dots,s_{n})\in\R^{n}|\text{$s_{i}=0$ or $1$ for some $i$}\}$ with the subdiffeology;
		\item $\partial_{\ep}\R^{n}=\{(s_{1},\dots,s_{n})\in\R^{n}|\text{$s_{i}<\ep$ or $s_{i}>1-\ep$ for some $i$}\}$ with the subdiffeology;
		\item $\mathbb{A}^{n}=\{(s_{0},\dots,s_{n})\in\R^{n+1}|\text{$\sum_{i}s_{i}=1$ for some $i$}\}$ with the subdiffeology;
		\item $\partial_{\ep}\mathbb{A}^{n}=\{(s_{0},\dots,s_{n})\in\mathbb{A}^{n}|\text{$s_{i}<\ep$ for some $i$}\}$ with the subdiffeology;
		\item $\diffsimplex{n}_{\mathsf{sub}}=\{(s_{0},\dots,s_{n})\in\mathbb{A}^{n}|\text{$0\leq s_{i}\leq1$ for any $i$}\}$ with the subdiffeology;
		\item $\partial_{\ep}\diffsimplex{n}_{\mathsf{sub}}=\{(s_{0},\dots,s_{n})\in\partial_{\ep}\mathbb{A}^{n}|\text{$0\leq s_{i}\leq1$ for any $i$}\}$ with the subdiffeology.
	\end{enumerate}
	Then, there exists a zig-zag of smooth maps
	\begin{equation*}
		\R^{n}
			\xrightarrow{\text{identity}}
		\R^{n}
			\xrightarrow{\Phi}
		\mathbb{A}^{n}
			\xleftarrow{\text{restriction}}
		\diffsimplex{n}_{\mathsf{sub}}
	\end{equation*}
	induces a sequence of isomorphisms
	\begin{equation*}
		\pi^{\mathcal{D}}_{n}(D,x)
			\cong
		[(\R^{n},\partial\R^{n}),(D,x)]
			\xleftarrow{\simeq}
		[(\R^{n},\partial_{\ep}\R^{n}),(D,x)]
			\xleftarrow{\simeq}
		[(\mathbb{A}^{n},\partial_{\ep}\mathbb{A}^{n}),(D,x)]
			\xrightarrow{\simeq}
		[(\diffsimplex{n}_{\mathsf{sub}},\partial_{\ep}\diffsimplex{n}_{\mathsf{sub}}),(D,x)],
	\end{equation*}
	where the smooth map $\Phi\colon\R^{n}\to\mathbb{A}^{n}$ is given by
	\begin{equation*}
		\Phi(s_{1},\dots,s_{n})
			\coloneqq
		(1-s_{1},s_{1}-s_{2},\dots,s_{n}-s_{n-1},s_{n}-0).
	\end{equation*}
\end{thm}
\begin{rem}
	This is a somewhat rewritten version of the Christensen-Wu theorem in a form for later use.
\end{rem}
\begin{proof}
	Although the map $\Phi$ is given in a different way, it can be proved by mimicking the proof of theorem 3.2 of Christensen-Wu \cite{christensen2015homotopy}.
\end{proof}

Kihara \cite{Kihara2018} introduces other diffeology of standard simplex $\topsimplex{n}=\{(s_{0},\dots,s_{n})\in[0,1]^{n+1}|\sum_{i}s_{i}=1\}$ for each $n\geq0$ in order to define a model structure on the category $\mathsf{Diff}$ of diffeological spaces which is Quillen equivalent to the model category $\sSet$ of simplicial sets with Kan-Quillen model structure. The diffeological standard simplices $\topsimplex{n}$ with Kihara's diffeology satisfy following properties:
\begin{enumerate}
	\item the underlying topological space $\mathsf{Dtplg}(\topsimplex{n})$ of $\topsimplex{n}$ is the topological standard $n$-simplex for any $n\geq0$;
	\item any affine map $\diffsimplex{m}\to\diffsimplex{n}$ is smooth:
	\item the $r$th horn $\diffhorn{r}{n}$ is a deformation retract of $\diffsimplex{n}$ for $n\geq 1$ and $0\leq r\leq n$;
	\item the canonical map $\diffsimplex{n}\to\diffsimplex{n}_{\mathsf{sub}}$ whose underlying map is identity is smooth;
	\item thhe above canonical map $\diffsimplex{1}\to\diffsimplex{1}_{\mathsf{sub}}$ is isomorphic if $n=1$.
\end{enumerate}
These diffeological simplices gives a Kan complex $S^{\mathcal{D}}(D)\coloneqq\Hom{\mathsf{Diff}}(\diffsimplex{\bullet},D)$ for any diffeological space $D$. In addition, it gives a Quillen equivalence $|-|_{\mathcal{D}}\dashv S^{\mathcal{D}}\colon\sSet\to\mathsf{Diff}$.
\begin{thm}[Kihara \cite{Kihara2018} Lemma 9.12 and Theorem 1.4]\label{Kihara1}
	Let $(D,x)$ be a pointed diffeological space. Then, there exists a natural bijection
	\begin{equation*}
		\pi^{\mathcal{D}}_{n}(D,x)
			\cong
		[(\diffsimplex{\bullet}_{\mathsf{sub}},\partial_{\ep}\diffsimplex{\bullet}_{\mathsf{sub}}),(D,x)]
			\xrightarrow{\simeq}
		[(\diffsimplex{\bullet},\partial_{\ep}\diffsimplex{\bullet}),(D,x)]
			\cong
		\pi_{n}(S^{\mathcal{D}}(D),x)
	\end{equation*}
	that is an isomorphism of groups for $n>0$.
\end{thm}
The family of smooth maps $\diffsimplex{\bullet}\to\diffsimplex{\bullet}_{\mathsf{sub}}\to\mathbb{A}^{n}$ induces a sequence of natural transformations
\begin{center}
\begin{tikzpicture}
	\node (01) at (0*3, 1*1) {$S^{\mathcal{D}}$};
	\node (11) at (1*3, 1*1) {$S^{\mathcal{D}}_{\mathsf{sub}}$};
	\node (21) at (2*3, 1*1) {$S^{\mathcal{D}}_{\mathsf{aff}}$};
	\node (00) at (0*3, 0*1) {$\Hom{\mathsf{Diff}}(\diffsimplex{\bullet},-)$};
	\node (10) at (1*3, 0*1) {$\Hom{\mathsf{Diff}}(\diffsimplex{\bullet}_{\mathsf{sub}},-)$};
	\node (20) at (2*3, 0*1) {$\Hom{\mathsf{Diff}}(\mathbb{A}^{\bullet},-)$};
	\draw[->, double, double distance=2pt, -implies] (11) -- (01);
	\draw[->, double, double distance=2pt, -implies] (21) -- (11);
	\draw[->, double, double distance=2pt, -implies] (10) -- (00);
	\draw[->, double, double distance=2pt, -implies] (20) -- (10);
	\draw[transform canvas={xshift= 1pt}] (01) -- (00);
	\draw[transform canvas={xshift=-1pt}] (01) -- (00);
	\draw[transform canvas={xshift= 1pt}] (11) -- (10);
	\draw[transform canvas={xshift=-1pt}] (11) -- (10);
	\draw[transform canvas={xshift= 1pt}] (21) -- (20);
	\draw[transform canvas={xshift=-1pt}] (21) -- (20);
\end{tikzpicture}.
\end{center}
It gives a sequance of simplicial sets $S^{\mathcal{D}}(D)\leftarrow S^{\mathcal{D}}_{\mathsf{sub}}(D)\leftarrow S^{\mathcal{D}}_{\mathsf{aff}}(D)$ for any diffeological space $D$.
\begin{thm}[Kihara \cite{Kihara2022smooth} Theorem 1.1]\label{Kihara2}
	The above canonical morphisms
	\begin{equation*}
		S^{\mathcal{D}}(D)
			\leftarrow
		S^{\mathcal{D}}_{\mathsf{sub}}(D)
			\leftarrow
		S^{\mathcal{D}}_{\mathsf{aff}}(D)
	\end{equation*}
	are weak homotopy equivalences for any diffeological space $D$. In particular, $S^{\mathcal{D}}(D)$ is a fibrant approximation both of $S^{\mathcal{D}}_{\mathsf{aff}}(D)$ and $S^{\mathcal{D}}_{\mathsf{sub}}(D)$.
\end{thm}

\subsubsection{smooth and continuous}
\begin{prop}\label{loop space comparison 1}
	Let $\mathcal{M}$ be a smooth manifold. Then unit map $\eta_{\mathcal{M}}\colon\mathcal{M}\to\mathsf{Cdiff}\mathsf{Dtplg}(\mathcal{M})$ is a weak equivalence.
\end{prop}
\begin{proof}
	From theorem \ref{Kihara1}, the following diagram commute for any $x\in\mathcal{M}$:
	\begin{center}
	\begin{tikzpicture}
		\node (01) at (0*4, 1*1) {$\pi_{\bullet}(S^{\mathcal{D}}(\mathcal{M}),x)$};
		\node (00) at (0*4, 0*1) {$\pi^{\mathcal{D}}_{\bullet}(\mathcal{M},x)$};
		\node (11) at (1*4, 1*1) {$\pi_{\bullet}(S^{\mathcal{D}}\mathsf{Cdiff}\mathsf{Dtplg}(\mathcal{M}),x)$};
		\node (10) at (1*4, 0*1) {$\pi_{\bullet}(\mathcal{M},x)$};
		\draw[->] (01) -- (11);
		\draw[->] (01) -- (00);
		\draw[->] (11) -- (10);
		\draw[->] (00) -- (10);
	\end{tikzpicture}
	\end{center}
	It is known that the map $\pi^{\mathcal{D}}_{\bullet}(\mathcal{M},x)\to\pi_{\bullet}(\mathcal{M},x)$ is isomorphic for any point $x\in\mathcal{M}$
\end{proof}
\begin{cor}
	Let $\mathcal{M}$ be a smooth manifold. Then the canonical map
	\begin{equation*}
		\Omega_{x}S^{\mathcal{D}}(\mathcal{M})
			\to
		\Omega_{x}S^{\mathcal{D}}\mathsf{Cdiff}\mathsf{Dtplg}(\mathcal{M})
	\end{equation*}
	is a weak equivalence for each point $x\in\mathcal{M}$.
\end{cor}
The functor $S^{\mathcal{D}}\mathsf{Cdiff}$ coincides with the functor $S\colon\mathsf{Top}\to\sSet$ which attaches the (topological) singular simplicial set to any topological spaces. Thus, there is a canonical weak equivalence
\begin{equation}
	\Omega_{x}S^{\mathcal{D}}(\mathcal{M})
		\to
	\Omega_{x}S(\mathcal{M})
\end{equation}
for any smooth manifold $\mathcal{M}$.

\subsubsection{diffeological and simplicial}
Let $(D,x)$ be a pointed diffeological space. For each non-negative integer $n\geq0$ and smooth map $\phi\colon\diffsimplex{n}\to\mathsf{Loop}(X,x)$, we define a smooth map $\diffsimplex{n}\times\diffsimplex{1}\to D$ as 
\begin{equation*}
	((s_{0},\dots,s_{n}),(t_{0},t_{1}))
		\mapsto
	\phi(s_{0},\dots,s_{n})(t_{1}).
\end{equation*}
The canonical map $|\simplex{}{n}\times\simplex{}{1}|\to\diffsimplex{n}\times\diffsimplex{1}$ gives a simplicial map $\simplex{}{n}\times\simplex{}{1}\to S^{\mathcal{D}}(D)$. In addition, we obtain a map $S^{\mathcal{D}}(\mathsf{Loop}(D,x))\to\Omega_{x}S^{\mathcal{D}}(D)$.

\begin{prop}\label{loop space comparison 2}
	Let $(\mathcal{M},x)$ be a pointed smooth manifold. Then, the canonical map
	\begin{equation*}
		S^{\mathcal{D}}(\mathsf{Loop}(\mathcal{M},x))
			\to
		\Omega_{x}S^{\mathcal{D}}(\mathcal{M})
	\end{equation*}
	is weak homotopy equivalence if the connected component contains $x$ is simply connected.
\end{prop}
\begin{proof}
	For any $\phi\colon(\mathbb{A}^{\bullet+1},\partial_{\ep}\mathbb{A}^{\bullet+1})\to(\mathcal{M},x)$, we define a smooth map $\partial^{\mathbb{A}}(\phi)\colon\mathbb{A}^{\bullet}\to\mathsf{Loop}(\mathcal{M},x)$ as
	\begin{equation*}
		\partial^{\mathbb{A}}(\phi)(s_{0},\dots,s_{\bullet})(t)
			\coloneqq
		\phi(s_{0},s_{1}\dots,s_{\bullet-1},s_{\bullet}-t,t).
	\end{equation*}
	It gives a map $\partial^{\mathbb{A}}\colon[(\mathbb{A}^{\bullet+1},\partial_{\ep}\mathbb{A}^{\bullet+1}),(\mathcal{M},x)]\to[(\mathbb{A}^{\bullet},\partial_{\ep}\mathbb{A}^{\bullet}),(\mathsf{Loop}(\mathcal{M},x),x)]$.
	From theorem \ref{Christensen-Wu} and theorem \ref{Kihara1}, we obtain the following commutative diagram.
	\begin{center}
	\begin{tikzpicture}
		\node (07) at (0*8, 7*1) {$\pi_{\bullet+1}^{\mathcal{D}}(\mathcal{M},x)$};
		\node (17) at (1*8, 7*1) {$\pi_{\bullet}^{\mathcal{D}}(\mathsf{Loop}(\mathcal{M},x),\mathrm{const}_{x})$};
		\node (06) at (0*8, 6*1) {$[(\R^{\bullet+1},\partial\R^{\bullet+1}),(\mathcal{M},x)]$};
		\node (16) at (1*8, 6*1) {$[(\R^{\bullet},\partial\R^{\bullet}),(\mathsf{Loop}(\mathcal{M},x),\mathrm{const}_{x})]$};
		\node (05) at (0*8, 5*1) {$[(\R^{\bullet+1},\partial_{\ep}\R^{\bullet+1}),(\mathcal{M},x)]$};
		\node (15) at (1*8, 5*1) {$[(\R^{\bullet},\partial_{\ep}\R^{\bullet}),(\mathsf{Loop}(\mathcal{M},x),\mathrm{const}_{x})]$};
		\node (04) at (0*8, 4*1) {$[(\mathbb{A}^{\bullet+1},\partial_{\ep}\mathbb{A}^{\bullet+1}),(\mathcal{M},x)]$};
		\node (14) at (1*8, 4*1) {$[(\mathbb{A}^{\bullet},\partial_{\ep}\mathbb{A}^{\bullet}),(\mathsf{Loop}(\mathcal{M},x),\mathrm{const}_{x})]$};
		\node (03) at (0*8, 3*1) {$[(\topsimplex{\bullet+1,\mathsf{sub}},\partial_{\ep}\topsimplex{\bullet+1,\mathsf{sub}}),(\mathcal{M},x)]$};
		\node (13) at (1*8, 3*1) {$[(\topsimplex{\bullet,\mathsf{sub}},\partial_{\ep}\topsimplex{\bullet,\mathsf{sub}}),(\mathsf{Loop}(\mathcal{M},x),\mathrm{const}_{x})]$};
		\node (02) at (0*8, 2*1) {$[(\topsimplex{\bullet+1},\partial_{\ep}\topsimplex{\bullet+1}),(\mathcal{M},x)]$};
		\node (12) at (1*8, 2*1) {$[(\topsimplex{\bullet},\partial_{\ep}\topsimplex{\bullet}),(\mathsf{Loop}(\mathcal{M},x),\mathrm{const}_{x})]$};
		\node (11) at (1*8, 1*1) {$\pi_{\bullet}(S^{\mathcal{D}}(\mathsf{Loop}(\mathcal{M},x)),\mathrm{const}_{x})$};
		\node (01) at (0*8, 1*1) {$\pi_{\bullet+1}(S^{\mathcal{D}}(\mathcal{M}),x)$};
		\node (10) at (1*8, 0*1) {$\pi_{\bullet}(\Omega_{x}S^{\mathcal{D}}(\mathcal{M}),\mathrm{const}_{x})$};
		\node (00) at (0*8, 0*1) {$\pi_{\bullet+1}(S^{\mathcal{D}}(\mathcal{M}),x)$};
		\draw[transform canvas={yshift= 1pt}] (07) --node[above] {$\scriptstyle \text{definition}$} (17);
		\draw[transform canvas={yshift=-1pt}] (07) -- (17);
		\draw[->] (04) --node[above] {$\scriptstyle \partial^{\mathbb{A}}$} (14);
		\draw[->] (00) --node[above] {$\scriptstyle \simeq$} (10);
		\path (00) --node[below] {$\scriptstyle \partial$} (10);
		\draw[transform canvas={xshift= 1pt}] (01) -- (00);
		\draw[transform canvas={xshift=-1pt}] (01) -- (00);
		\draw[->] (06) --node[right] {$\scriptstyle \simeq$} (07);
		\draw[->] (16) --node[right] {$\scriptstyle \simeq$} (17);
		\draw[->] (05) --node[right] {$\scriptstyle \simeq$} (06);
		\path (05) --node[left] {$\scriptstyle \text{cut off}$} (06);
		\draw[->] (15) --node[right] {$\scriptstyle \simeq$} (16);
		\path (15) --node[left] {$\scriptstyle \text{cut off}$} (16);
		\draw[->] (04) --node[right] {$\scriptstyle \simeq$} (05);
		\draw[->] (14) --node[right] {$\scriptstyle \simeq$} (15);
		\draw[transform canvas={xshift= 1pt}] (04) -- (03);
		\draw[transform canvas={xshift=-1pt}] (04) -- (03);
		\draw[transform canvas={xshift= 1pt}] (14) -- (13);
		\draw[transform canvas={xshift=-1pt}] (14) -- (13);
		\draw[->] (03) --node[right] {$\scriptstyle \simeq$} (02);
		\draw[->] (13) --node[right] {$\scriptstyle \simeq$} (12);
		\draw[->] (02) --node[right] {$\scriptstyle \simeq$} (01);
		\draw[->] (12) --node[right] {$\scriptstyle \simeq$} (11);
		\draw[->] (11) -- (10);
%		\node at (4, 5.5) {$\circlearrowright$};
	\end{tikzpicture}
	\end{center}
	
	Since the connected component of $\mathcal{M}$ contains $x$ is simply connected, we obtain the following commutative diagram for any smooth loop $\gamma\in\mathsf{Loop}(\mathcal{M},x)$;
	\begin{center}
	\begin{tikzpicture}
		\node (01) at (0*5, 1*1) {$\pi_{\bullet}(S^{\mathcal{D}}(\mathsf{Loop}(\mathcal{M},x)),\gamma)$};
		\node (11) at (1*5, 1*1) {$\pi_{\bullet}(\Omega_{x}S^{\mathcal{D}}(X),\gamma)$};
		\node (00) at (0*5, 0*1) {$\pi_{\bullet}(S^{\mathcal{D}}(\mathsf{Loop}(\mathcal{M},x)),\mathrm{const}_{x})$};
		\node (10) at (1*5, 0*1) {$\pi_{\bullet}(\Omega_{x}S^{\mathcal{D}}(X),\mathrm{const}_{x})$};
		\draw[->] (00) --node[above] {$\scriptstyle \simeq$} (10);
		\draw[->] (00) --node[right] {$\scriptstyle \simeq$} (01);
		\draw[->] (10) --node[right] {$\scriptstyle \simeq$} (11);
		\draw[->] (01) -- (11);
	\end{tikzpicture}
	\end{center}
	Therefore, the canonical map $S^{\mathcal{D}}(\mathsf{Loop}(\mathcal{M},x))\to\Omega_{x}S^{\mathcal{D}}(D)$ is weak homotopy equivalence.
\end{proof}
\begin{rem}[a viewpoint via homotopy limits]
	The above canonical map is given by the universal property of (homotopy) pullback.
	\begin{center}
	\begin{tikzpicture}
		\node (14) at (1*3, 4*1) {$S^{\mathcal{D}}(D)$};
		\node (11) at (1*3, 1*1) {$S^{\mathcal{D}}(D\times D)$};
		\node (13) at (1*3, 3*1) {$S^{\mathcal{D}}(\mathsf{Path}(D))$};
		\node (02) at (0*3, 2*1) {$S^{\mathcal{D}}(\mathsf{Loop}(D,x))$};
		\node (00) at (0*3, 0*1) {$S^{\mathcal{D}}(\ast)$};
		\node (34) at (3*3, 4*1) {$S^{\mathcal{D}}(\mathcal{M})$};
		\node (31) at (3*3, 1*1) {$S^{\mathcal{D}}(D)\times S^{\mathcal{D}}(D)$};
		\node (33) at (3*3, 3*1) {$S^{\mathcal{D}}(D)^{\simplex{}{1}}$};
		\node (22) at (2*3, 2*1) {$\Omega_{x}S^{\mathcal{D}}(D)$};
		\node (20) at (2*3, 0*1) {$\simplex{}{0}$};
		\draw[right hook->] (34) --node[right] {$\scriptstyle \sim$} (33);
		\draw[->>] (33) -- (31);
		\draw[->] (14) -- (13);
		\draw[->] (13) -- (11);
		\draw[transform canvas={yshift= 1pt}] (14) -- (34);
		\draw[transform canvas={yshift=-1pt}] (14) -- (34);
		\draw[->] (11) --node[above] {$\scriptstyle \simeq$} (31);
		\draw[->] (00) -- (11);
		\draw[->] (20) -- (31);
		\draw[->] (02) -- (00);
		\draw[->] (02) -- (13);
		\draw[->] (13) -- (33);
		\draw[->] (00) -- (20);
		\draw[white, line width=2pt] (22) -- (20);
		\draw[->>] (22) -- (20);
		\draw[->] (22) -- (33);
		\draw[white, line width=2pt] (02) -- (22);
		\draw[->, dashed] (02) -- (22);
	\end{tikzpicture}
	\end{center}
	Since the functor $S^{\mathcal{D}}\colon\mathsf{Diff}\to\sSet$ is right Quillen functor, it follows that the canonical map is weak homotopy equivalence if the projection $\mathsf{Path}(D)\to D\times D$ is a fibration.
\end{rem}

\subsection{comparison of de Rham algebras}\label{review dR}
In this subsection, we consider a comparison of ``de Rham algebras''.

First, we review the de Rham algebra of diffeological spaces in the sense of Souriau. A plot $p\colon U\to D$ of diffeological space $D$ is a ``smooth'' map from open subset $U$ of the Euclidean space. We can regard it as an approximation of ``smooth space $D$'' by an open subset of Euclidean space. In view of this, the de Rham algebra $\mathcal{A}_{\mathsf{dR}}(D)$ of $D$ (in the sense of Souriau) is defined as an approximation by commutative cochain algebra $\mathcal{A}(U)$, that is as a limit $\varprojlim_{U}\mathcal{A}(U)$ (as commutative cochain algebra).
\begin{center}
\begin{tikzpicture}
	\node (00) at (0*2, 0*0.5) {$V$};
	\node (02) at (0*2, 2*0.5) {$U$};
	\node (11) at (1*2, 1*0.5) {$D$};
	\draw[->] (02) --node[above] {$\scriptstyle p$} (11);
	\draw[->] (00) --node[below] {$\scriptstyle q$} (11);
	\draw[->] (00) --node[left] {$\scriptstyle \phi$} (02);
	\node (30) at (3*2, 0*0.5) {$\mathcal{A}_{\mathsf{dR}}(V)$};
	\node (32) at (3*2, 2*0.5) {$\mathcal{A}_{\mathsf{dR}}(U)$};
	\node (41) at (4*2, 1*0.5) {$\mathcal{A}_{\mathsf{dR}}(D)$};
	\draw[<-] (32) --node[above] {$\scriptstyle p^{\ast}$} (41);
	\draw[<-] (30) --node[below] {$\scriptstyle q^{\ast}$} (41);
	\draw[<-] (30) --node[left] {$\scriptstyle \phi^{\ast}$} (32);
\end{tikzpicture}
\end{center}
Iglesias-Zemmour has introduced an integration
\begin{equation*}
	\int^{\mathsf{IZ}}
		\colon
	\mathcal{A}_{\mathsf{dR}}(D)
		\to
	\mathcal{C}_{\mathsf{cube}}(D)
\end{equation*}
which is a cochain map. It is a generalization of the ``de Rham map'' for smooth manifolds. 
However, the cochain map is generally \bf not\rm{} quasi-isomorphism. In other words, the de Rham theorem does \bf not\rm{} hold (in the sense of Souriau).

As de Rham algebra such that the de Rham theorem holds, Kuribayashi \cite{kuribayashi2020simplicial} has introduced a new de Rham algebra. At this time, we review it. Let $\mathcal{A}_{\mathsf{aff}}$ be the simplicial cochain algebra defined by $(\mathcal{A}_{\mathsf{aff}})_{n}\coloneqq\mathcal{A}_{\mathsf{dR}}(\mathbb{A}^{n})$ for each $n\geq0$. Since simplicial cochain algebra gives a contravariant functor from the category $\sSet$ of simplicial sets to the category $\mathsf{CAlg}$ of commutative cochain algebras. Therefore, we obtain a contravariant functor $\mathcal{A}_{\mathsf{aff}}(S^{\mathcal{D}}_{\mathsf{aff}}(-))$ from the category $\mathsf{Diff}$ of diffeological spaces to the category $\mathsf{CAlg}$ of commutative cochain algebras.
\begin{thm}[Kuribayashi \cite{kuribayashi2020simplicial} Theorem 2.4]\label{Kuribayashi de Rham}
	For a diffeological space $D$, one has a homotopy commutative diagram
	\begin{center}
	\begin{tikzpicture}
		\node (00) at (0*5, 0*1.25) {$\displaystyle \mathcal{A}_{\mathrm{dR}}(D)$};
		\node (10) at (1*5, 0*1.25) {$\displaystyle \mathcal{C}_{\mathsf{cube}}(D)$};
		\node (01) at (0*5, 1*1.25) {$\mathcal{A}_{\mathrm{aff}}(S^{\mathcal{D}}_{\mathsf{aff}}(D))$};
		\node (11) at (1*5, 1*1.25) {$\mathcal{C}(S^{\mathcal{D}}_{\mathsf{sub}}(D),\R)$};
		\node (02) at (0*5, 2*1.25) {$\displaystyle (\mathcal{C}_{\mathsf{PL}}\otimes\mathcal{A}_{\mathrm{aff}})(S^{\mathcal{D}}_{\mathsf{aff}}(D))$};
		\node (12) at (1*5, 2*1.25) {$\displaystyle \mathcal{C}(S^{\mathcal{D}}_{\mathsf{sub}}(D),\R)$};
		\draw[->] (00) --node[left] {$\scriptstyle \alpha$} (01);
		\draw[->] (10) --node {$\scriptstyle \sim$} (11);
		\draw[->] (01) --node {$\scriptstyle \sim$} (02);
		\draw[->] (01) --node[left] {$\scriptstyle \psi$} (02);
		\draw[transform canvas={xshift= 1pt}] (12) -- (11);
		\draw[transform canvas={xshift=-1pt}] (12) -- (11);
		\draw[->] (02) -- (11);
		\draw[->] (12) --node[above] {$\scriptstyle \sim$} (02);
		\draw[->] (12) --node[below] {$\scriptstyle \varphi$} (02);
		\draw[->] (01) --node[above] {$\scriptstyle \sim$} (11);
		\path (01) --node[below] {$\scriptstyle \text{an ``integration'' $\int$}$} (11);
		\draw[->] (00) --node[below] {$\scriptstyle \int^{\mathsf{IZ}}$} (10);
	\end{tikzpicture}
	\end{center}
	in which $\varphi$ amd $\psi$ are quasi-isomorphisms of cochain algebras and the integration map $\int$ is a morphism of cochain complexes.
\end{thm}
\begin{rem}
	Kuribayashi \cite{kuribayashi2020simplicial} also has shown that the factor map $\alpha$ is a quasi-isomorphism if $D$ is a ``good'' diffeological space. 
\end{rem}

Finally, we compare two versions of de Rham algebras, the one is introduced in Kuribayashi \cite{kuribayashi2020simplicial} and another one is introduced in \cite{kageyama2022higher}.
We define two simplicial commutative cochain algebras $\mathcal{A}_{\mathrm{PP}}$ and $\mathcal{A}_{\mathrm{PS}}$ as
\begin{align*}
	(\mathcal{A}_{\mathrm{PP}})_{n}
		&\coloneqq
	\dR{\simplex{}{n}}{\bullet}{,\Z\langle\dq\rangle}\otimes_{\Z\langle \dq,x_{1}\dots,x_{n}\rangle}\R[x_{1},\dots,x_{n}],\\
	(\mathcal{A}_{\mathrm{PS}})_{n}
		&\coloneqq
	\dR{\simplex{}{n}}{\bullet}{,\Z\langle\dq\rangle}\otimes_{\Z\langle \dq,x_{1}\dots,x_{n}\rangle}C^{\infty}(\R^{n}).
\end{align*}
\begin{center}
\begin{tikzpicture}
	\node (02) at (0*2.5, 2*0.8) {$\Z\langle \dq,x_{1}\dots,x_{n}\rangle$};
	\node (22) at (2*2.5, 2*0.8) {$\R[x_{1}\dots,x_{n}]$};
	\node (42) at (4*2.5, 2*0.8) {$C^{\infty}(\R^{n})$};
	\node (00) at (0*2.5, 0*0.8) {$\dR{\simplex{}{n}}{\bullet}{,\Z\langle\dq\rangle}$};
	\node (20) at (2*2.5, 0*0.8) {$(\mathcal{A}_{\mathrm{PP}})_{n}$};
	\node (40) at (4*2.5, 0*0.8) {$(\mathcal{A}_{\mathrm{PS}})_{n}$};
	\node (13) at (1*2.5, 3*0.8) {$\Z\langle \dq,x_{1}\dots,x_{m}\rangle$};
	\node (33) at (3*2.5, 3*0.8) {$\R[x_{1}\dots,x_{m}]$};
	\node (53) at (5*2.5, 3*0.8) {$C^{\infty}(\R^{m})$};
	\node (11) at (1*2.5, 1*0.8) {$\dR{\simplex{}{m}}{\bullet}{,\Z\langle\dq\rangle}$};
	\node (31) at (3*2.5, 1*0.8) {$(\mathcal{A}_{\mathrm{PP}})_{m}$};
	\node (51) at (5*2.5, 1*0.8) {$(\mathcal{A}_{\mathrm{PS}})_{m}$};
	\draw[->] (13) -- (11);
	\draw[->] (13) -- (33);
	\draw[->] (33) -- (53);
	\draw[->] (33) -- (31);
	\draw[->] (53) -- (51);
	\draw[->] (11) -- (31);
	\draw[->] (31) -- (51);
	\draw[->] (13) -- (02);
	\draw[->] (33) -- (22);
	\draw[->] (53) -- (42);
	\draw[->] (11) -- (00);
	\draw[->] (31) -- (20);
	\draw[->] (51) -- (40);
	\draw[white, line width=2pt] (02) -- (22);
	\draw[->] (02) -- (22);
	\draw[white, line width=2pt] (22) -- (42);
	\draw[->] (22) -- (42);
	\draw[->] (02) -- (00);
	\draw[white, line width=2pt] (22) -- (20);
	\draw[->] (22) -- (20);
	\draw[white, line width=2pt] (42) -- (40);
	\draw[->] (42) -- (40);
	\draw[->] (00) -- (20);
	\draw[->] (20) -- (40);
\end{tikzpicture}
\end{center}
Then, there are canonical morphisms of simplicial cochain algebras
\begin{equation*}
	\mathcal{A}_{\mathrm{PS}}
		\leftarrow
	\mathcal{A}_{\mathrm{PP}}
		\leftarrow
	\dR{\simplex{}{-}}{\bullet}{,\Z\langle\dq\rangle}.
\end{equation*}
Clearly, the simplicial cochain algebras $\mathcal{A}_{\mathsf{PS}}$ coincied with $\mathcal{A}_{\mathsf{aff}}$.

\subsection{comparison theorem}\label{comparison}
Combining the above comparisons, we obtain the following theorem.
\begin{thm}\label{comparison1}
	Let $D$ be a diffeological space. Then, there is a morphism
	\begin{equation*}
		\mathbb{J}
			\colon
%		\R\bigotimes^{}_{\mathcal{A}_{\mathrm{PS}}(S^{\mathcal{D}}_{\mathsf{aff}}(\mathcal{M}))\otimes\mathcal{A}_{\mathrm{PS}}(S^{\mathcal{D}}_{\mathsf{aff}}(\mathcal{M}))}\bm{B}\mathcal{A}_{\mathrm{PS}}(S^{\mathcal{D}}_{\mathsf{aff}}(\mathcal{M}))
		\mathsf{CC}(\mathcal{A}_{\mathrm{PS}}(S^{\mathcal{D}}_{\mathsf{aff}}(\mathcal{M})),\R)
			\to
		\mathcal{A}_{\mathrm{PS}}(S^{\mathcal{D}}_{\mathsf{aff}}(\Omega_{x,\mathcal{D}}\mathcal{M}))
	\end{equation*}
	of commutative cochain algebras for which the following diagram is commutative:
	\begin{center}
	\begin{tikzpicture}
		\node (36) at (3*1.5, 6*1.25) {$\Z\langle\dq\rangle\bigotimes^{}_{\dR{S(\mathcal{M})}{}{,\Z\langle\dq\rangle}\otimes\dR{S(\mathcal{M})}{}{,\Z\langle\dq\rangle}}\bm{B}\dR{S(\mathcal{M})}{}{,\Z\langle\dq\rangle}$};
		\node (76) at (7*1.5+1, 6*1.25) {$\dR{\Omega_{x}S(\mathcal{M})}{}{,\Z\langle\dq\rangle}$};
		\draw[->] (36) --node[above] {$\scriptstyle \mathbb{I}$} (76);
		\node (34) at (3*1.5, 4*1.25) {$\mathsf{CC}(\mathcal{A}_{\mathsf{PS}}(S(\mathcal{M})),\R)$};
		\node (74) at (7*1.5+1, 4*1.25) {$\mathcal{A}_{\mathsf{PS}}(\Omega_{x}S(\mathcal{M}))$};
		\draw[->] (34) --node[above] {$\scriptstyle \mathbb{I}$} (74);
		\draw[->] (36) -- (34);
		\draw[->] (76) -- (74);
		\node (25) at (2*1.5, 5*1.25) {$\Z\langle\dq\rangle\bigotimes^{}_{\dR{S^{\mathcal{D}}(\mathcal{M})}{}{,\Z\langle\dq\rangle}\otimes\dR{S^{\mathcal{D}}(\mathcal{M})}{}{,\Z\langle\dq\rangle}}\bm{B}\dR{S^{\mathcal{D}}(\mathcal{M})}{}{,\Z\langle\dq\rangle}$};
		\node (65) at (6*1.5+1, 5*1.25) {$\dR{\Omega_{x}S^{\mathcal{D}}(\mathcal{M})}{}{,\Z\langle\dq\rangle}$};
		\draw[->] (25) --node[above] {$\scriptstyle \mathbb{I}$} (65);
		\node (23) at (2*1.5, 3*1.25) {$\mathsf{CC}(\mathcal{A}_{\mathsf{PS}}(S^{\mathcal{D}}(\mathcal{M})),\R)$};
		\node (63) at (6*1.5+1, 3*1.25) {$\mathcal{A}_{\mathsf{PS}}(\Omega_{x}S^{\mathcal{D}}(\mathcal{M}))$};
		\draw[->] (23) --node[above] {$\scriptstyle \mathbb{I}$} (63);
		\draw[white, line width=2pt] (25) -- (23);
		\draw[white, line width=2pt] (65) -- (63);
		\draw[->] (25) -- (23);
		\draw[->] (65) -- (63);
		\node (12) at (1*1.5, 2*1.25) {$\mathsf{CC}(\mathcal{A}_{\mathsf{PS}}(S^{\mathcal{D}}(\mathcal{M})),\R)$};
		\node (52) at (5*1.5+1, 2*1.25) {$\mathcal{A}_{\mathsf{PS}}(S^{\mathcal{D}}(\mathsf{Loop}(\mathcal{M},x)))$};
		\node (01) at (0*1.5, 1*1.25) {$\mathsf{CC}(\mathcal{A}_{\mathsf{PS}}(S^{\mathcal{D}}_{\mathsf{aff}}(\mathcal{M})),\R)$};
		\node (41) at (4*1.5+1, 1*1.25) {$\displaystyle \mathcal{A}_{\mathsf{PS}}(S^{\mathcal{D}}_{\mathsf{aff}}(\mathsf{Loop}(\mathcal{M},x)))$};
		\draw[->] (01) --node[above] {$\scriptstyle \mathbb{J}$} (41);
		\node (00) at (0*1.5, 0*1.25) {$\mathsf{CC}(\mathcal{A}_{\mathsf{dR}}(\mathcal{M}),\R)$};
		\node (40) at (4*1.5+1, 0*1.25) {$\mathcal{A}_{\mathsf{dR}}(\mathsf{Loop}(\mathcal{M},x))$};
		\draw[->] (00) --node[above] {$\scriptstyle \mathsf{C}$} (40);
		\draw[->] (00) --node {$\scriptstyle \sim$} (01);
		\path (00) --node[left] {$\scriptstyle \mathsf{exp}(\alpha)$} (01);
		\draw[->] (40) --node[left] {$\scriptstyle \alpha$} (41);
		\draw[->] (12) --node {$\scriptstyle \sim$} (01);
		\draw[->] (52) --node {$\scriptstyle \sim$} (41);
		\draw[->] (23) --node {$\scriptstyle \sim$} (12);
		\draw[->] (63) --(52);
%		\draw[->] (73) --node[fill=white] {$\scriptstyle \sim$} (62);
		\draw[->] (34) --node {$\scriptstyle \sim$} (23);
		\draw[->] (74) --node {$\scriptstyle \sim$} (63);
		\draw[->] (36) --node {$\scriptstyle \sim$} (25);
		\draw[->] (76) --node {$\scriptstyle \sim$} (65);
	\end{tikzpicture}
	\end{center}
	where $\mathsf{CC}(\mathcal{A},\R)\coloneqq\R\bigotimes^{}_{\mathcal{A}\otimes\mathcal{A}}\bm{B}\mathcal{A}$ for each cochain algebra $\mathcal{A}$.
\end{thm}
\begin{proof}
	Simplicial integrations can be defined in the same way even if polynomial functions are replaced by smooth functions. In addition, we can show that the smooth version of lemma \ref{Int Lem21} holds in the same way. Thus, for any differential form $\omega\colon\simplex{}{r}\times X\to\mathcal{A}_{\mathrm{PS}}$, we can define the simplicial fiberwise integration
	\begin{equation*}
		(\proj{\simplex{}{r}})_{\ast}\omega
			\colon
		X
			\to
		\mathcal{A}_{\mathrm{PS}}.
	\end{equation*}

	Let $(\mathcal{M},x)$ be a pointed smooth manifold. For each non-negative integer $r\geq 0$, we obtain a.simplicial map $\Phi(\mathcal{M})^{r}_{i}\colon\simplex{}{r}\to S^{\mathcal{D}}_{\mathsf{aff}}(\mathcal{M})$ as a composition of simplicial maps
	\begin{equation*}
		\simplex{}{r}\times S^{\mathcal{D}}_{\mathsf{aff}}(\mathsf{Loop}(\mathcal{M},x))
			\to
		\simplex{}{1}^{r}\times S^{\mathcal{D}}_{\mathsf{aff}}(\mathsf{Loop}(\mathcal{M},x))^{r}
			\to
		(S^{\mathcal{D}}_{\mathsf{aff}}(\R)\times S^{\mathcal{D}}_{\mathsf{aff}}(\mathsf{Loop}(\mathcal{M},x)))^{r}
			\to
		S^{\mathcal{D}}_{\mathsf{aff}}(\mathcal{M})^{r}
			\xrightarrow{\proj{i}}
		S^{\mathcal{D}}_{\mathsf{aff}}(\mathcal{M})
	\end{equation*}
	where $S^{\mathcal{D}}_{\mathsf{aff}}(\R)\times S^{\mathcal{D}}_{\mathsf{aff}}(\mathsf{Loop}(\mathcal{M},x))\to S^{\mathcal{D}}_{\mathsf{aff}}(\mathcal{M})$ is given by
	\begin{equation*}
		(\mathbb{A}^{n}\xrightarrow{\gamma}\R,\mathbb{A}^{n}\xrightarrow{f}\mathsf{Loop}(\mathcal{M},x))
			\mapsto
		\begin{bmatrix}
			\mathbb{A}^{n}&\xrightarrow{f(-)(\gamma(-))}&\mathcal{M}\\
			(t_{0},\dots,t_{n})&\mapsto&f(t_{0},\dots,t_{n})(\gamma(t_{0},\dots,t_{n}))
		\end{bmatrix}.
	\end{equation*}
	Then we obtain a homogeneous differential form
	\begin{equation*}
		\omega_{1}\times\dots\times\omega_{r}
			\coloneqq
		(\omega_{1}\Phi^{r}_{1})\wedge\dots\wedge(\omega_{r}\Phi^{r}_{r})
			\colon
		\simplex{}{r}\times S^{\mathcal{D}}_{\mathsf{aff}}(\mathsf{Loop}(\mathcal{M},x))
			\to
		\mathcal{A}_{\mathrm{PS}}
	\end{equation*}
	for any pair of homogeneous differential forms $\omega_{1},\dots,\omega_{r}\colon S^{\mathcal{D}}(\mathcal{M})\to\mathcal{A}_{\mathrm{PS}}$. Using the above integration, we define
	\begin{equation*}
		\mathbb{J}([\omega_{1}|\dots|\omega_{r}])
			\coloneqq
		(-1)^{\sum_{i=1}^{r}(r-i)(|\omega_{i}|-1)}(\proj{\simplex{}{r}})_{\ast}(\omega_{1}\times\dots\times\omega_{r}).
	\end{equation*}
	It gives a morphism
	\begin{equation*}
		\mathbb{J}
			\colon
		\mathsf{CC}(\mathcal{A}_{\mathrm{PS}}(S^{\mathcal{D}}_{\mathsf{aff}}(\mathcal{M}));\R)
			\to
		\mathcal{A}_{\mathrm{PS}}(S^{\mathcal{D}}_{\mathsf{aff}}(\Omega_{x,\mathcal{D}}\mathcal{A}))
	\end{equation*}
	of cochain algebra.
	
	Chen's iterated integral assigns a homogeneous differential form $\int\omega_{1}\dots\omega_{r}$ on loop space for any pair of homogeneous differential forms $\omega_{1},\dots,\omega_{r}$ on $\mathcal{M}$. For each smooth map $\varphi^{\vee}\colon\mathbb{A}^{n}\times\R\to\mathcal{M}$ satisfying $\varphi^{\vee}(s_{0},\dots,s_{n},0)=x$ and $\varphi^{\vee}(s_{0},\dots,s_{n},1)=x$ for any $(s_{0},\dots,s_{n})\in\mathbb{A}^{n}$, the above differential form gives a differential form on $\mathbb{A}^{n}$ given by
	\begin{equation*}
		(-1)^{\sum_{i=1}^{r}(r-i)(|\omega_{i}|-1)}\varphi^{\ast}(\int\omega_{1}\cdots\omega_{r})
			\coloneqq
		(-1)^{\sum_{i=1}^{r}(r-i)(|\omega_{i}|-1)}(\proj{\topsimplex{r}})_{\ast}(\id{\topsimplex{r}}\times\varphi)^{\ast}(\omega_{1}\times\cdots\times\omega_{r}).
	\end{equation*}
	It is an $n$-simplex of $\mathcal{A}_{\mathsf{PS}}$.
	On the other hands, $\mathbb{J}$ gives an $n$-simplex
	\begin{align*}
		\mathbb{J}\mathsf{exp}(\alpha)([\omega_{1}|\cdots|\omega_{r}])(\varphi)
			&=
		\mathbb{J}([\alpha(\omega_{1})|\cdots|\alpha(\omega_{r})])(\varphi)\\
			&=
		(-1)^{\sum_{i=1}^{r}(r-i)(|\omega_{i}|-1)}(\proj{\simplex{}{r}})_{\ast}(\id{\simplex{}{r}}\times\varphi)^{\ast}(\alpha_{\mathcal{M}}(\omega_{1})\times\cdots\times\alpha_{\mathcal{M}}(\omega_{r}))\\
			&=
		(-1)^{\sum_{i=1}^{r}(r-i)(|\omega_{i}|-1)}(\proj{\topsimplex{r}})_{\ast}(\id{\topsimplex{r}}\times\varphi)^{\ast}(\omega_{1}\times\cdots\times\omega_{r})\\
			&=
		(-1)^{\sum_{i=1}^{r}(r-i)(|\omega_{i}|-1)}\varphi^{\ast}(\int\omega_{1}\cdots\omega_{r})\\
			&=
		\alpha\mathsf{C}([\omega_{1}|\cdots|\omega_{r}])(\varphi)
	\end{align*}
	of $\mathcal{A}_{\mathsf{PS}}$.
\end{proof}
By combining theorem \ref{comparison} and theorem \ref{Kuribayashi de Rham}, we obtain the following theorem.
\begin{thm}\label{comparison2}
	Let $D$ be a diffeological space. Then, there is the following homotopy commutative diagram:
	\begin{center}
	\begin{tikzpicture}
		\node (00) at (0*2.5, 0*1.5) {$\displaystyle \mathsf{CC}(\mathcal{A}_{\mathrm{dR}}(\mathcal{M}),\R)$};
		\node (20) at (2*2.5+1, 0*1.5) {$\displaystyle \mathcal{A}_{\mathrm{dR}}(\mathsf{Loop}(\mathcal{M},x))$};
		\node (40) at (4*2.5, 0*1.5) {$\displaystyle \mathcal{C}_{\mathsf{cube}}(\mathsf{Loop}(\mathcal{M},x))$};
		\node (01) at (0*2.5, 1*1.5) {$\mathsf{CC}(\mathcal{A}_{\mathrm{PS}}(S^{\mathcal{D}}_{\mathsf{aff}}(\mathcal{M})),\R)$};
		\node (21) at (2*2.5+1, 1*1.5) {$\mathcal{A}_{\mathrm{PS}}(S^{\mathcal{D}}_{\mathsf{aff}}(\mathsf{Loop}(\mathcal{M},x)))$};
		\node (41) at (4*2.5, 1*1.5) {$\mathcal{C}(S^{\mathcal{D}}_{\mathsf{sub}}(\mathsf{Loop}(\mathcal{M},x)),\R)$};
		\node (12) at (1*2.5, 2*1.5) {$\mathsf{CC}(\mathcal{A}_{\mathrm{PS}}(S(\mathcal{M})),\R)$};
		\node (32) at (3*2.5+1, 2*1.5) {$\displaystyle \mathcal{A}_{\mathrm{PS}}(\Omega_{x}S(\mathcal{M}))$};
		\node (52) at (5*2.5, 2*1.5) {$\displaystyle \mathcal{C}(\Omega_{x}S(\mathcal{M}),\R)$};
		\node (13) at (1*2.5, 3*1.5) {$\Z\langle\dq\rangle\bigotimes^{}_{\dR{S(\mathcal{M})}{}{,\Z\langle\dq\rangle}\otimes\dR{S(\mathcal{M})}{}{,\Z\langle\dq\rangle}}\bm{B}\dR{S(\mathcal{M})}{}{,\Z\langle\dq\rangle}$};
		\node (33) at (3*2.5+1, 3*1.5) {$\displaystyle \dR{\Omega_{x}S(\mathcal{M})}{}{,\Z\langle\dq\rangle}$};
		\node (53) at (5*2.5, 3*1.5) {$\displaystyle \mathcal{C}(\Omega_{x}S(\mathcal{M}),\Z\langle\dq\rangle)$};
		\draw[->] (00) --node {$\scriptstyle \sim$} (01);
		\draw[->] (20) --node[left] {$\scriptstyle \alpha$} (21);
		\draw[->] (40) --node {$\scriptstyle \sim$} (41);
		\draw[->] (12) --node {$\scriptstyle \sim$} (01);
		\draw[->] (32) -- (21);
		\draw[->] (52) -- (41);
		\draw[->] (13) -- (12);
		\draw[->] (33) -- (32);
		\draw[->] (53) -- (52);
		\draw[->] (13) --node[below] {$\scriptstyle \mathbb{I}$} (33);
		\draw[->] (33) -- (53);
		\draw[->] (12) --node[below] {$\scriptstyle \mathbb{I}$} (32);
		\draw[->] (32) -- (52);
		\draw[->] (01) --node[below] {$\scriptstyle \mathbb{J}$} (21);
		\draw[->] (21) -- (41);
		\draw[->] (00) --node[below] {$\scriptstyle \mathsf{C}$} (20);
		\draw[->] (20) -- (40);
	\end{tikzpicture}
	\end{center}
\end{thm}
\section{A relationship with Chen's theorem and Hain's theorem}
Let $(\mathcal{M},x)$ be a pointed smooth manifold. Using iterated integral, we obtain an isomorphism
\begin{equation*}
	\widehat{\R[\pi_{1}(\mathcal{M},x)]}
		\xrightarrow{\simeq}
	H^{0}(\R\bigotimes^{\bm{L}}_{\mathcal{A}_{\mathsf{dR}}(\mathcal{M})\otimes\mathcal{A}_{\mathsf{dR}}(\mathcal{M})}\mathcal{A}_{\mathsf{dR}}(\mathcal{M}))^{\ast}
\end{equation*}
of Hopf algebras and a morphism
\begin{equation*}
	\bigoplus_{r>1}\R\otimes\pi_{r+1}(\mathcal{M},x)
		\to
	\bigoplus_{r\geq0}H^{r}(\R\bigotimes^{\bm{L}}_{\mathcal{A}_{\mathsf{dR}}(\mathcal{M})\otimes\mathcal{A}_{\mathsf{dR}}(\mathcal{M})}\mathcal{A}_{\mathsf{dR}}(\mathcal{M}))^{\ast}
\end{equation*}
of graded $\R$-vector space. If $\mathcal{M}$ is simply connected, the later map induces an isomorphism
\begin{equation*}
	\bigoplus_{r>1}\R\otimes\pi_{r+1}(\mathcal{M},x)
		\to
	\mathsf{P}\bigoplus_{r\geq0}H^{r}(\R\bigotimes^{\bm{L}}_{\mathcal{A}_{\mathsf{dR}}(\mathcal{M})\otimes\mathcal{A}_{\mathsf{dR}}(\mathcal{M})}\mathcal{A}_{\mathsf{dR}}(\mathcal{M}))^{\ast}
\end{equation*}
of graded Lie algebra where the left-hand side is a graded Lie algebra whose bracket is Whitehead product and the right-hand side is a graded Lie algebra of primitive elements of graded Hopf algebra. The former morphism is a result of Chen, the latter by Hain. These morphisms are given by composition
\begin{center}
\begin{tikzpicture}
	\node (03) at (0*6, 3*1+0.5) {$\pi_{n+1}(\mathcal{M},x)$};
	\node (00) at (0*6, 0*1) {$\pi_{n}(\Omega_{x}\mathcal{M},x)$};
	\node (10) at (1*6, 0*1) {$H_{n}(\Omega_{x}\mathcal{M};\R)$};
	\node (11) at (1*6, 1*1) {$H^{n}(S^{\mathcal{D}}_{\mathsf{sub}}(\mathsf{Loop}(\mathcal{M},x));\R)^{\ast}$};
	\node (12) at (1*6, 2*1) {$H^{n}(\mathcal{A}_{\mathrm{dR}}(\mathsf{Loop}(\mathcal{M},x)))^{\ast}$};
	\node (22) at (2*6, 2*1) {$H^{n}(\mathsf{CC}(\mathcal{A}_{\mathrm{dR}}(\mathcal{M}),\R))^{\ast}$};
	\node (23) at (2*6, 3*1+0.5) {$\displaystyle H^{n}(\R\bigotimes^{\bm{L}}_{\mathcal{A}_{\mathsf{dR}}(\mathcal{M})\otimes\mathcal{A}_{\mathsf{dR}}(\mathcal{M})}\mathcal{A}_{\mathsf{dR}}(\mathcal{M}))^{\ast}$};
	\draw[->] (03) --node[right] {$\scriptstyle \simeq$} (00);
	\path (03) --node[left] {$\scriptstyle \partial_{n}$} (00);
	\draw[->] (00) --node[below] {$\scriptstyle h_{n}$} (10);
	\draw[->] (10) --node[left] {$\scriptstyle \chi$} (11);
	\draw[->] (11) --node[left] {$\scriptstyle \int^{\ast}$} (12);
	\draw[->] (12) --node[below] {$\scriptstyle \mathsf{C}^{\ast}$} (22);
	\draw[transform canvas={xshift= 1pt}] (22) -- (23);
	\draw[transform canvas={xshift=-1pt}] (22) -- (23);
	\draw[->, dashed] (03) -- (23);
\end{tikzpicture}
\end{center}
where 
\begin{itemize}
	\item $\partial\colon\pi_{n+1}(\mathcal{M},x)\xrightarrow{\simeq}\pi_{n}(\Omega_{x}\mathcal{M},x)$ is the connecting homomorphism,
	\item $h_{n}\colon\pi_{n}(\Omega_{x}\mathcal{M},x)\to H_{n}(\Omega_{x}\mathcal{M})$ is the Hurewicz map,
	\item $\chi\colon H_{n}(\Omega_{x}\mathcal{M};\R)\to H^{n}(S^{\mathcal{D}}_{\mathsf{sub}}(\mathsf{Loop}(\mathcal{M},x));\R)^{\ast}$ is a canonical map into double-dual,
	\item $\int^{\ast}\colon H^{n}(S^{\mathcal{D}}_{\mathsf{sub}}(\mathsf{Loop}(\mathcal{M},x));\R)^{\ast}\to H^{n}(\mathcal{A}_{\mathrm{dR}}(\mathsf{Loop}(\mathcal{M},x)))^{\ast}$ is (the dual of) the de Rham map,
	\item and $\mathsf{C}^{\ast}\colon H^{n}(\mathcal{A}_{\mathrm{dR}}(\mathsf{Loop}(\mathcal{M},x)))^{\ast}\to H^{n}(\mathsf{CC}(\mathcal{A}_{\mathrm{dR}}(\mathcal{M}),\R))^{\ast}$ is (the dual of) the iterated integral.
\end{itemize}
\begin{rem}
	Chen and Hain use Chen's formal homology connection $(\omega,\delta)$ to represent the derived tensor $\R\bigotimes^{\bm{L}}_{\mathcal{A}_{\mathsf{dR}}(\mathcal{M})\otimes\mathcal{A}_{\mathsf{dR}}(\mathcal{M})}\mathcal{A}_{\mathsf{dR}}(\mathcal{M})$ where $\delta$ is a derivation $\delta\colon\widehat{\mathsf{T}H_{+}(\mathcal{M})[-1]}\to\widehat{\mathsf{T}H_{+}(\mathcal{M})[-1]}$ and $\omega$ is a ``degree $1$'' element $\omega\in\bigoplus_{r}\widehat{\mathsf{T}H_{+}(\mathcal{M})[-1]}_{r}\otimes\mathcal{A}^{r+1}(\mathcal{M})$. Using it, the above morpshisms gives morphisms
	\begin{align*}
		\widehat{\R[\pi_{1}(\mathcal{M},x)]}
			&\to
		H_{0}(\widehat{\mathsf{T}H_{+}(\mathcal{M})[-1]},\delta),\\
		\R\otimes\pi_{\bullet+1}(\mathcal{M},x)
			&\to
		H_{\bullet}(\widehat{\mathsf{T}H_{+}(\mathcal{M})[-1]},\delta).
	\end{align*}
\end{rem}

For any ponited smooth manifold $(\mathcal{M},x)$, we obtain a following commutative diagram by theorem \ref{comparison2}:
\begin{center}
\begin{tikzpicture}
	\node (00) at (0*2.5, 0*1.5) {$\displaystyle H^{0}(\mathsf{CC}(\mathcal{A}_{\mathrm{dR}}(\mathcal{M}),\R))$};
	\node (10) at (2*2.5, 0*1.5) {$\displaystyle H^{0}(\mathcal{A}_{\mathrm{dR}}(\mathsf{Loop}(\mathcal{M},x)))$};
	\node (20) at (4*2.5, 0*1.5) {$H^{0}(S^{\mathcal{D}}_{\mathsf{sub}}(\mathsf{Loop}(\mathcal{M},x)),\R)$};
	\node (01) at (0*5, 1*1.5) {$H^{0}(\mathsf{CC}(\mathcal{A}_{\mathrm{PS}}(S(\mathcal{M})),\R))$};
	\node (11) at (1*5, 1*1.5) {$\displaystyle H^{0}(\mathcal{A}_{\mathrm{PS}}(\Omega_{x}S(\mathcal{M})))$};
	\node (21) at (2*5, 1*1.5) {$\displaystyle H^{0}(\Omega_{x}S(\mathcal{M}),\R)$};
	\draw[transform canvas={xshift= 1pt}] (00) -- (01);
	\draw[transform canvas={xshift=-1pt}] (00) -- (01);
	\draw[transform canvas={xshift= 1pt}] (20) -- (21);
	\draw[transform canvas={xshift=-1pt}] (20) -- (21);
	\draw[->] (01) --node[above] {$\scriptstyle \mathbb{I}$} (11);
	\draw[->] (11) -- (21);
	\draw[->] (00) --node[above] {$\scriptstyle \mathsf{C}$} (10);
	\draw[->] (10) -- (20);
\end{tikzpicture}.
\end{center}
In addition, if $\mathcal{M}$ is simply connected, we obtain the following commutative diagram by theorem \ref{loop space comparison 2}:
\begin{center}
\begin{tikzpicture}
	\node (00) at (0*2.5, 0*1.5) {$\displaystyle H^{\bullet}(\mathsf{CC}(\mathcal{A}_{\mathrm{dR}}(\mathcal{M}),\R))$};
	\node (10) at (2*2.5, 0*1.5) {$\displaystyle H^{\bullet}(\mathcal{A}_{\mathrm{dR}}(\mathsf{Loop}(\mathcal{M},x)))$};
	\node (20) at (4*2.5, 0*1.5) {$H^{\bullet}(S^{\mathcal{D}}_{\mathsf{sub}}(\mathsf{Loop}(\mathcal{M},x)),\R)$};
	\node (01) at (0*5, 1*1.5) {$H^{\bullet}(\mathsf{CC}(\mathcal{A}_{\mathrm{PS}}(S(\mathcal{M})),\R))$};
	\node (11) at (1*5, 1*1.5) {$\displaystyle H^{\bullet}(\mathcal{A}_{\mathrm{PS}}(\Omega_{x}S(\mathcal{M})))$};
	\node (21) at (2*5, 1*1.5) {$\displaystyle H^{\bullet}(\Omega_{x}S(\mathcal{M}),\R)$};
	\draw[transform canvas={xshift= 1pt}] (00) -- (01);
	\draw[transform canvas={xshift=-1pt}] (00) -- (01);
	\draw[transform canvas={xshift= 1pt}] (20) -- (21);
	\draw[transform canvas={xshift=-1pt}] (20) -- (21);
	\draw[->] (01) --node[above] {$\scriptstyle \mathbb{I}$} (11);
	\draw[->] (11) -- (21);
	\draw[->] (00) --node[above] {$\scriptstyle \mathsf{C}$} (10);
	\draw[->] (10) -- (20);
\end{tikzpicture}.
\end{center}
Hence it can be seen that a ``smooth structure'' is not essentially necessary for Chen's theorem and Hain's theorem.

Let $\K$ be an arbitrary commutative ring. Then, there is a sequence of $\infty$-categories
\begin{center}
\begin{tikzpicture}
	\node (0) at (0*2.5, 0) {$\mathscr{S}$};
	\node (1) at (1*2.5, 0) {$\mathit{c}\mathscr{CA}\mathit{lg}(\K)$};
	\node (2) at (2*2.5, 0) {$\mathscr{CA}\mathit{lg}(\K)\opposite$};
	\draw[->] (0) --node[above] {$\scriptstyle \K[-]$} (1);
	\draw[->] (1) -- (2);
	\draw[->] (0) -- (0*2.5, -0.5) --node[below] {$\scriptstyle \mathscr{A}$} (2*2.5, -0.5) -- (2);
\end{tikzpicture}
\end{center}
where $\mathscr{S}$ is an $\infty$-category of spaces, $\mathit{c}\mathscr{CA}\mathit{lg}(\K)$ is an $\infty$-category of ``cocomutative coalgebras over $\K$'' and $\mathscr{CA}\mathit{lg}(\K)$ is an $\infty$-category of ``commutative algebras over $\K$''. More precisely, $\mathit{c}\mathscr{CA}\mathit{lg}(\K)$ is an $\infty$-category of $\mathbb{E}_{\infty}$-coalgebras over $\K$ and $\mathscr{CA}\mathit{lg}(\K)$ is an $\infty$-category of $\mathbb{E}_{\infty}$-algebras over $\K$. These $\infty$-functors give canonical maps
\begin{equation*}
	\Hom{\mathscr{S}}(\mathit{pt},\{x\}\underset{X\times X}{\times}X)
		\to
	\Hom{\mathscr{CA}\mathit{lg}(\K)\opposite}(\mathscr{A}(\mathit{pt}),\mathscr{A}(\{x\}\underset{X\times X}{\times}X))
		\to
	\Hom{\mathscr{CA}\mathit{lg}(\K)\opposite}(\mathscr{A}(\mathit{pt}),\mathscr{A}(\{x\})\underset{\mathscr{A}(X)\times\mathscr{A}(X)}{\times}\mathscr{A}(X)).
\end{equation*}
From the previous observations, it is expected that this map is a lift of the map of Chen's theorem and Hain's theorem when $\K=\R$. This is a claim that has not yet been given proper proof and is one for important future works.

\end{document}